\documentclass[a4paper, 11pt]{amsart}

\usepackage{graphicx}
\graphicspath{ {figures} }
\usepackage[]{amsmath}
\usepackage{overpic}
\usepackage{xcolor}
\usepackage{caption}
\usepackage{enumerate}
\usepackage{paralist}
\usepackage{tikz}
\usetikzlibrary{intersections, calc}

\newcommand{\N}{{\mathbb N}}
\newcommand{\Z}{{\mathbb Z}}
\newcommand{\R}{{\mathbb R}}
\newcommand{\C}{{\mathbb C}}

\renewcommand{\ij}{\ensuremath i\!j}

\newtheorem{theorem}{Theorem}[section]
\newtheorem{proposition}[theorem]{Proposition}
\newtheorem{lemma}[theorem]{Lemma}
\newtheorem{definition}[theorem]{Definition}
\newtheorem{corollary}[theorem]{Corollary}

\newcommand{\edge}{\bullet\hspace{-.6ex}-\hspace{-.6ex}\bullet}

\newlength{\subfigureheight}
\makeatletter
\@fpsep\textheight
\makeatother

\title{Orthogonal ring patterns in the plane}
\author{Alexander I.~Bobenko, Tim Hoffmann, and Thilo R\"orig}

\address{Alexander I. Bobenko, Thilo R\"orig\vspace{-.5em}}
\address{Institute of Mathematics, Secr. MA 8-4, TU Berlin, 10623 Berlin, Germany\vspace{-.5em}}
\address{Email: {\normalfont \{bobenko,roerig\}@math.tu-berlin.de}}
\address{Tim Hoffmann\vspace{-.5em}}
\address{Dept. of Mathematics, TU Munich, 85748 Garching, Germany\vspace{-.5em}}
\address{Email: {\normalfont tim.hoffmann@ma.tum.de}}
\begin{document}

\begin{abstract}
  We introduce orthogonal ring patterns consisting of pairs of concentric
  circles generalizing circle patterns. 
  We show that orthogonal ring patterns are governed by the same equation as
  circle patterns.
  For every ring pattern there exists a one parameter family of patterns that
  interpolates between a circle pattern and its dual. 
  We construct ring patterns analogues of the Doyle spiral, Erf and $z^\alpha$
  functions.
  We also derive a variational principle and compute ring patterns based on
  Dirichlet and Neumann boundary conditions. 
\end{abstract}

\maketitle

\section{Introduction}
\label{sec:introduction}

The theory of circle patterns can be seen as a discrete version of conformal
maps. Schramm~\cite{schramm1997} has studied orthogonal circle patterns on
the~$\Z^2$-lattice, has proven their convergence to conformal maps and
constructed discrete analogs of some entire holomorphic functions. Circle
patterns are described by a variational
principle~\cite{bobenko_springborn_variational}, which is given in terms of
volumes of ideal hyperbolic polyhedra~\cite{hyperbolic_polyhedra}.  We
introduce orthogonal ring patterns that are natural generalizations of circle
patterns. Our theory of orthogonal ring patterns has its origin in discrete
differential geometry of S-isothermic cmc surfaces
\cite{bobenko_hoffmann_Scmc_2016}.  Recently, orthogonal double circle patterns
(ring patterns) on the sphere have been used to construct discrete
surfaces S-cmc by Tellier et al.~\cite{tellier_etal_2018}.  

We start Sect.~\ref{sec:ring_patterns} with a definition of orthogonal ring
patterns and their elementary properties. In particular we show that all rings
have the same area.  Our main Theorem~\ref{thm:ring_pattern} shows that ring
patterns are described by an equation for variables at the vertices.
Furthermore, each ring pattern comes with a natural $1$-parameter family of
patterns.  In Sect.~\ref{sec:circle_patterns} we show that as the area of the
rings goes to zero the ring patterns converge to orthogonal circle patterns.
In the following Sect.~\ref{sec:examples} we introduce ring patterns analogs of
Doyle spirals, the Erf function, $z^\alpha$ for $\alpha\in (0,2]$, and the
logarithm.  Finally, we introduce a variational principle to construct ring
patterns for given Dirichlet or Neumann boundary conditions.  A remarkable fact
that we explore is that the orthogonal ring and circle patterns in $ {\mathbb
R}^2$ are governed by the same integrable equation.  In a subsequent
publication we plan to develop a theory of ring patterns in a sphere and
hyperbolic space. They are governed by equations in elliptic functions that belong 
to the class of discrete integrable systems classified
in~\cite{adler_bobenko_suris_2003}.

\section{Orthogonal ring patterns}
\label{sec:ring_patterns}

In this section, we will introduce orthogonal ring patterns and show
that the existence of such the patterns is governed by the same equation as
the existence of orthogonal circle patterns.

We will consider cell complex~$G$ defined by a subset of the
quadrilaterals of the~$\Z^2$ lattice in~$\R^2$.  The vertices $V(G)$ of the complex~$G$
are indexed by $(m,n) \in \Z^2$ and denoted by $v_{m,n}$. %We assume that $G$ is simply connected. 
Each of its inner vertices has four neighbors, the vertices with less neighbors are called boundary vertices.
The vertices of the dual cell complex $G^*$ are identified with the 2-cells of $G$,
%small add to the version finalGD2
and the edges of $G^*$ correspond to the inner edges of $G$, i.e. to the edges shared by two neighboring 2-cells.
% and are indexed by $(m+\frac{1}{2},n+\frac{1}{2})$.
We assume that $G$ and $G^*$ are simply connected. 
%Thus any vertex of $G$ is adjacent to at least two edges and any edge of $G$ is adjacent to at least one quadrilateral.
The oriented edges
are given by pairs of vertices and are either \emph{horizontal} $(v_{m,n},
v_{m+1,n})$ or \emph{vertical} $(v_{m,n}, v_{m,n+1})$.

A \emph{ring} is a pair of two concentric circles in~$\R^2$ that form a
ring (annulus). We identify the vertices with the centers and denote the
inner circle and its radius by small letters~$c$ and $r$, and the outer circle
and its radius by capital letters $C$ and $R$.  We assign an orientation to the
ring by allowing $r$ to be negative: positive radius corresponds to
counter-clockwise and negative radius to clockwise orientation.  The outer
radius will always be positive. The area of a ring is given by
$(R^2-r^2)\pi$. Subscripts are used to associate circles and radii to vertices
of the complex, e.g., $c_{m,n}$ is the inner circle associated with the
vertex~$v_{m,n}$.

\begin{definition}[Orthogonal ring patterns]
  \label{def:ring_pattern}
  %Let~$G$ be a subcomplex of the~$\Z^2$-lattice defined by its quadrilaterals.
  An \emph{orthogonal ring pattern} consists of rings associated to the
  vertices of~$G$ satisfying the following properties:
  \begin{enumerate}
    \item \label{def:rp_intersection}
      The rings associated to neighboring vertices $v_i$ and $v_j$
      \emph{intersect orthogonally}, i.e., the outer circle $C_i$ of the one
      vertex intersects the inner circle $c_j$ of the other vertex orthogonally
      and vice versa (see Fig.~\ref{fig:circle_ring_pattern}, left).
    \item \label{def:rp_touch}
      In each square of~$G$ the inner circles $c_{m,n}$ and $c_{m+1,n+1}$ and
      the outer circles $C_{m,n+1}$ and $C_{m+1,n}$ pass through one point.
      (Then orthogonality implies that the two inner and the two outer circles
      touch in this point. see Fig.~\ref{fig:circle_ring_pattern}, center).
    \item \label{def:rp_orientation}
    For any ring $(C_{m,n}, c_{m,n})$ the four touching points $C_{m,n}\cap C_{m+1,n-1}$, $c_{m,n}\cap c_{m+1,n+1}$, $C_{m,n}\cap C_{m-1,n+1}$ and $c_{m,n}\cap c_{m-1,n-1}$ have the same orientation as~$c_{m,n}$, i.e., are in
      counter-clockwise order if $r_{m,n}$ is positive and in clockwise order
      if $r_{m,n}$ is negative.
  \end{enumerate}
\end{definition}

\begin{figure}[t]
  \centering
  \includegraphics[width=.9\linewidth]{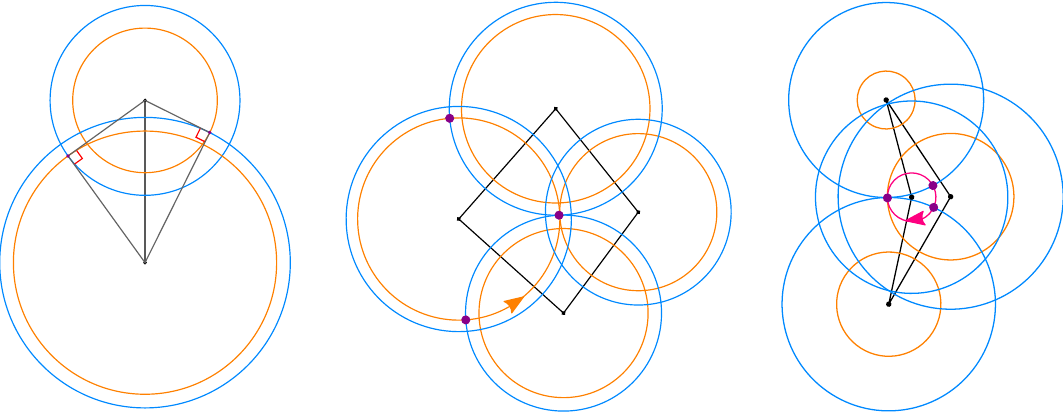}
  \caption{Left: Two orthogonally intersecting rings.
    Center: The inner circles touch along one diagonal of a quadrilateral and the
    outer circles along the other diagonal. The touching point coincides.
    Right: If the orientation (i.e., signed radii) of the inner circles differ,
    then the centers lie on the same side of the common tangent.
  }
  \label{fig:circle_ring_pattern}
\end{figure}

The orthogonal intersection of neighboring rings has the following implication
for their areas.

\begin{lemma}
  Consider two rings with radii $r_i,R_i$ and $r_j,R_j$ that 
  intersect orthogonally.
  Then the two rings have the same area.
\end{lemma}
\begin{proof}
  By Pythagoras' Theorem the square of the distance~$d$ between the circle
  centers is $R_i^2 + r_j^2 = d^2 = r_i^2 + R_j^2$ 
  since the inner and outer circles are intersecting orthogonally. 
  This equation is equivalent to the equality of the ring
  areas $(R_i^2 - r_i^2)\pi = (R_j^2 - r_j^2)\pi$.
\end{proof}

The constant area allows us to use a single variable $\rho_i$ to express the
inner and the outer radii of the rings in the following way: 
Consider an orthogonal ring pattern with constant ring area $A_0 = \pi
\ell_0^2$, that is, for the radii $r_i, R_i$ of all vertices~$v_i \in V(G)$ we
have $R_i^2 - r_i^2 = \ell_0^2$.
Then for each vertex we can choose a single variable $\rho_i$ by setting 
\begin{equation}
  \label{eq:rho_radii}
  R_i = \ell_0 \cosh(\rho_i) 
  \quad \text{and} \quad 
  r_i = \ell_0 \sinh(\rho_i)\,.
\end{equation}

We will call those new variables~\emph{$\rho$-radii}. The orientation of the
rings is encoded in the sign of the $\rho$-radii.  
In Sect.~\ref{sec:circle_patterns} we consider the limit of orthogonal
ring patterns as the area goes to zero.  The $\rho$-radii become
the logarithmic radii of a Schramm type orthogonal circle pattern~\cite{schramm1997}
in the limit.

As in the case of orthogonal circle patterns there exist families of  
vertices $V_e = \{(m,n) \in \Z^2 \,|\, m + n \text{ even} \}$ and 
$V_o = \{(m,n) \in \Z^2 \,|\, m + n \text{ odd} \}$ such that all
rings along the diagonals touch (see Fig.~\ref{fig:touching_rings}).

\begin{figure}[tb]
  \centering
  \includegraphics[width=.6\linewidth]{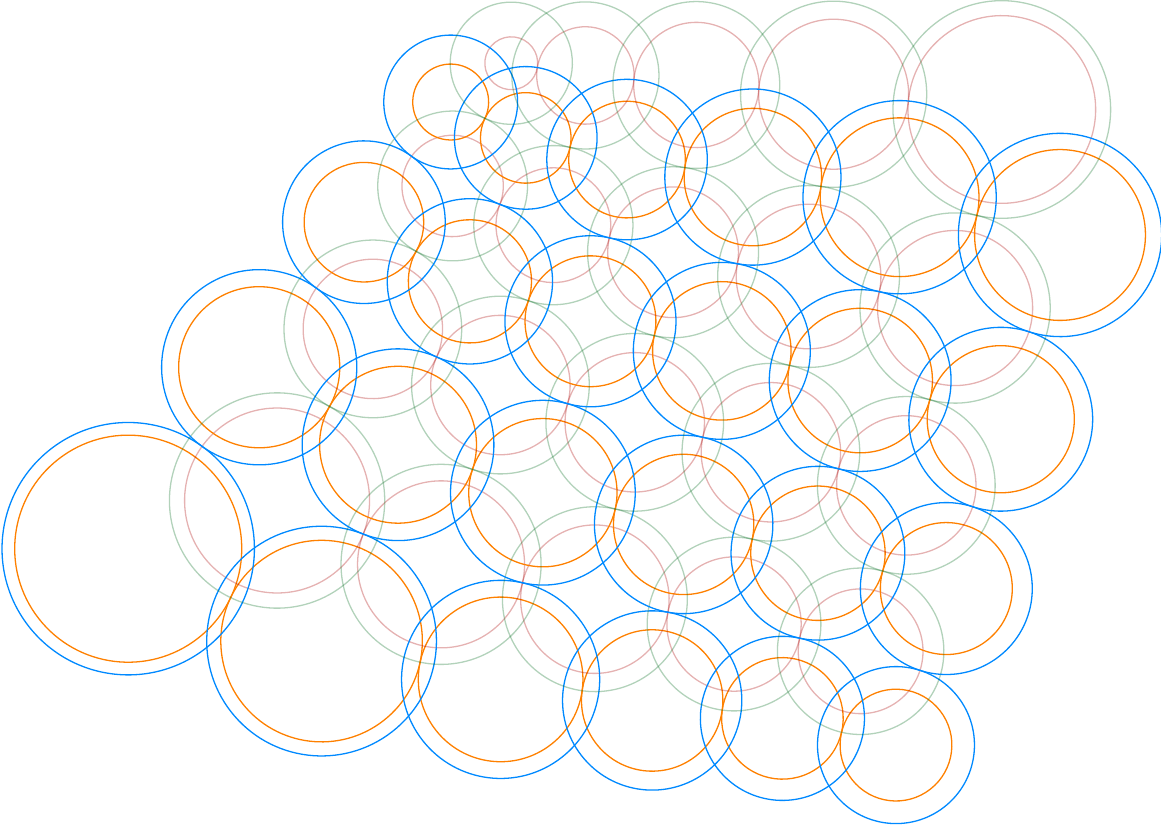}
  \caption{The rings of an orthogonal ring pattern partition into
  two diagonal families of touching rings.}
  \label{fig:touching_rings}
\end{figure}

Neighboring vertices of an orthogonal ring pattern define a cyclic
quadrilaterals of the following forms:

The circles $C_i, c_i$ and $C_j, c_j$ intersect in four points. 
Since the inner circle $c_i$ (resp.~$c_j$) and the outer circle $C_j$
(resp.~$C_i$) intersect orthogonally the centers of the circles and the
intersection points $c_i \cap C_j$ and $C_i \cap c_j$ lie on a circle.
We introduce four possible circular quadrilaterals,  shown in Fig.~\ref{fig:two_rings}, depending on the  
 orientation of the rings (i.e. on the sings of the $\rho$-radii). 
 Note that, the angle at the vertex~$v_i$ has the same sign as the
corresponding~$\rho_i$.

If $\rho_i=0$ the inner circle $c_i$ shrinks to its center and the cyclic quadrilateral defined 
by the rings $(C_i, c_i)$ and $(C_j, c_j)$ degenerates to a triangle
with a double vertex. The circle $C_j$ passes through this point.

Given the $\rho$-radii we can compute the angles in the cyclic quadrilaterals.
We will assume that the $\arctan$ function maps to oriented angles in
$(-\frac\pi2,\frac\pi2)$.

\begin{figure}[bt]
  \centering
  
  \begin{overpic}[width=\linewidth]{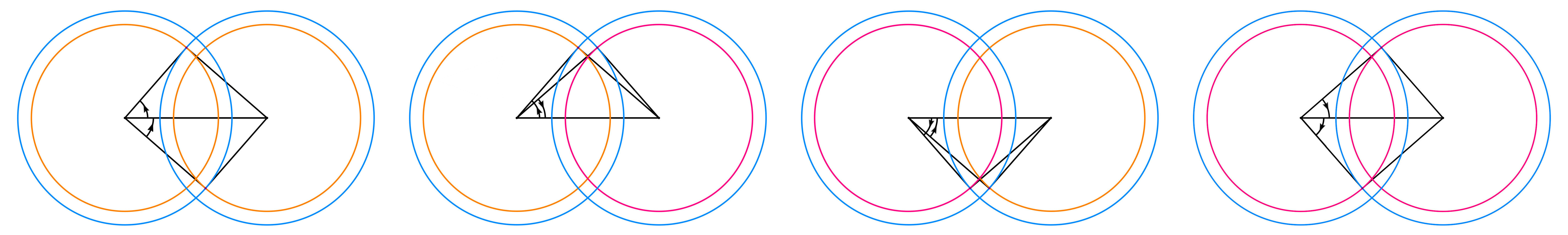}
    %\put(10,43){$0 > r_i$}
    %\put(23,25){$R_i$}
    %\put(55,45){$r_j > 0$}
    %\put(48,35){$R_j$}
  \end{overpic}
  \caption{
    Cyclic quadrilaterals defined by two orthogonally intersecting circle
    rings depending on the signs of the radii: \\
    (Left): $\rho_i>0,\rho_j > 0$, embedded quadrilateral, $\varphi_{\ij}>0$,\\
    (Center-Left): $\rho_i>0,\rho_j <0$, non-embedded quadrilateral, $\varphi_{\ij}>0$,\\
    (Center-Right): $\rho_i<0,\rho_j >0$, non-embedded quadrilateral, $\varphi_{\ij}<0$,\\
    (Right): $\rho_i<0,\rho_j <0$, embedded quadrilateral, $\varphi_{\ij}<0$.
  }
  \label{fig:two_rings}
\end{figure}

\begin{lemma}
  \label{lem:kite_angle}
  Let $v_i$ and $v_j$ be two neighboring vertices in an orthogonal ring
  pattern with $\rho$-radii $\rho_i$ and $\rho_j$.
  Then the angle at the vertex $v_i$ in the quadrilateral (triangle if $\rho_i=0$)
  defined by the two rings at $v_i$ and $v_j$
  %$(v_i, p , v_j, q)$ with 
  %$p \in c_i \cap C_j$ and $q \in c_j \cap C_i$ 
  is given by 
  \begin{equation}
  \label{eq:angle_ij}
    \varphi_{\ij} = 
    \begin{cases}
      \pi - 2 \arctan(e^{\rho_i - \rho_j}) & \text{if $\rho_i > 0$}\\
      \frac{\pi}{2}-2 \arctan(e^{- \rho_j}) & \text{if $\rho_i = 0$}\\
      - 2 \arctan(e^{\rho_i - \rho_j}) & \text{if $\rho_i < 0$}.
    \end{cases}
\end{equation}
\end{lemma}

\begin{proof}
For $\rho_i\neq 0$ the angle $\varphi_{\ij}$ is built by two angles of two rectangular triangles
$$
\varphi_{\ij}=\arg(1+i\frac{R_j}{r_i})+\arg(1+i\frac{r_j}{R_i})=\arg\left((1+i\frac{\cosh \rho_j}{\sinh \rho_i})(1+i\frac{\sinh \rho_j}{\cosh \rho_i})\right).
$$
Simple transformations of hyperbolic functions yield
\begin{eqnarray*}
\varphi_{\ij}&=&\arg \left(1- \frac{\sinh 2\rho_j}{\sinh 2\rho_i} +i\frac{2\cosh (\rho_i+\rho_j)}{\sinh 2\rho_i} \right)=\\
 & &\arg \left( \text{sign}\ (\rho_i) (\sinh (\rho_i-\rho_j)\cosh (\rho_i+\rho_j) +i\cosh (\rho_i+\rho_j)) \right)=\\
& &\arg \left( \text{sign}\ (\rho_i)  (i+\sinh (\rho_i-\rho_j))\right).
\end{eqnarray*}
Further, using 
$$
\arg(1+i\sinh x)=\arctan \sinh x= 2 \arctan e^x - \frac{\pi}{2},
$$
we arrive at the representations (\ref{eq:angle_ij}) for all $\rho_j$.

The angle $\varphi_{\ij}$ is discontinuous at $\rho_i=0$, and its value jumps by $\pi$:
$$
\varphi_{\ij}(\rho_i=0+)=\varphi_{\ij}(\rho_i=0-)+\pi .
$$
For $\rho_i=0$ the circle $c_i$ degenerates to a point located at the center of $C_i$, and the circle $C_j$ passes through this point. The quadrilateral degenerates to a triangle, and the angle of this triangle at the vertex $v_i$ is
$$
\varphi_{\ij}(\rho_i=0)=\arg(1+i\frac{r_j}{R_i})=\arg(1+i\sinh\rho_j)= \frac{\pi}{2}-2 \arctan(e^{- \rho_j}).
$$ 
\end{proof}

We define a cone angle at $v_i$ as the sum of the angles built by the ring centered at $v_i$ with all its neighbors:
$$
\Theta_i :=\sum_{j: v_j{\edge}v_i} \varphi_{\ij}.
$$ 
For interior vertices of an orthogonal ring pattern we have  
\begin{equation}
  \label{eq:Theta}
    \Theta_{i} = 
    \begin{cases}
      2\pi  & \text{if $\rho_i > 0$}\\
      0  & \text{if $\rho_i = 0$}\\
      - 2\pi & \text{if $\rho_i < 0$}.
    \end{cases}
\end{equation}
For a boundary vertex $\Theta_i>0$ if it is positively oriented $\rho_i > 0$, and $\Theta_i<0$ if it is negatively oriented $\rho_i < 0$.

%With the above angles we are able to prove the following theorem on orthogonal
%ring patterns.

\begin{theorem}[Orthogonal ring patterns]
  \label{thm:ring_pattern}
An orthogonal ring pattern $\mathcal{R}$ with simply connected $G$ and $G^*$ is uniquely determined by its $\rho$-radii function $\rho:V(G)\to {\mathbb R}$.   
  
A function $\rho:V(G)\to {\mathbb R}$ describes the $\rho$-radii of an orthogonal ring pattern on $G$ with the boundary cone angles $\Theta_i$ if and only if it satisfies:
 \begin{equation}
  \label{eq:closure}
    \sum_{j: v_j{\edge}v_i} 2 \arctan(e^{\rho_i - \rho_j}) = 
    \begin{cases}
      2\pi  & \text{for interiour vertices}\\
      \pi\ \rm{Val}(i)-\Theta_i & \text{for boundary vertex with $\rho_i > 0$}\\
      -\Theta_i & \text{for boundary vertex with $\rho_i < 0$}.
    \end{cases}
\end{equation} 
 Here the sum is taken over all neighboring vertices of  $v_i$, and $\rm{Val}(i)$ is the number of rings neighboring to the boundary ring $i$. 
\end{theorem}

\begin{proof}
The first claim of the theorem follows from the fact that a pair of orthogonal rings is determined by their $\rho$-radii uniquely up to Euclidean motion. Consequently laying the rings we obtain a simply connected ring pattern.

  Let $v_i\in V(G)$ be an interior vertex with four neighboring vertices $v_1, v_2, v_3$, and $v_4$. 
  The five rings form a flower in the pattern if and only if the angles $\varphi_{\ij}$ for 
  $j \in \{1,2,3,4\}$ sum up to $2 \pi$ (or $-2\pi$, depending on the orientation). 
 
 By Lemma~\ref{lem:kite_angle} for  positive $\rho_i$ the sum of the angles $\varphi_{\ij}$ around $v_i$ is $2\pi$ if
  \[
    2\pi 
    = \sum_{j=1}^4 \varphi_{\ij}
    = \sum_{j=1}^4 
      \pi - 2 \arctan(e^{\rho_i - \rho_j}).
  \]
  This is equivalent to~\eqref{eq:closure}.  
  For negative $\rho_i$  the other equation of Lemma~\ref{lem:kite_angle} also implies \eqref{eq:closure}.
  Hence we can assemble the four quadrilaterals and rings around the
  vertex~$v_i$ to form an orthogonal ring pattern. 
  As the complex~$G$ is simply connected the local proof suffices to prove that
  the entire complex~$G$ can be assembled to build an orthogonal ring
  pattern.
  
 The $\rho$-radii satisfy the same equation (\ref{eq:closure}) for the cases $\rho_i>0$ and $\rho_i<0$.
  This equation is also satisfied for $\rho_i=0$. This can be seen as the limit $\rho_i\to 0$ since the right hand side of (\ref{eq:closure}) is a continuous function of $\rho_i$. Alternatively, when the quadrilaterals degenerate to triangles the angles of the triangles at the vertex $v_i$ are given by (\ref{eq:angle_ij}) in the case $\rho_i=0$. Summing up around $v_i$ and using $\Theta_i=0$ we arrive at the same equation~\eqref{eq:closure}.

Formulas for the cone angles at the boundary rings follow directly from \eqref{eq:angle_ij}.
 \end{proof} 

The angle condition at the vertices of Thm.~\ref{thm:ring_pattern} only depends 
on the differences of the logarithmic radii.
So without violating equation~\eqref{eq:closure}, we can apply a shift $\rho \to
\rho^\delta = \rho + \delta$ by $\delta \in \R$ to the $\rho$-variables.

\begin{corollary}
  \label{cor:pattern_deformation}
  Consider an orthogonal ring pattern~$\mathcal{R}$ of area~$\pi$ for
  given $\rho$-radii $\rho_i$. 
  Then the $\rho$-radii $\rho_i^\delta = \rho_i + \delta$ define a one
  parameter family of orthogonal ring patterns~$\mathcal{R}^\delta$
  with radii:
  \begin{align*}
    r^{\delta}_i &= \sinh (\rho_i + \delta) \\
    R^{\delta}_i &= \cosh (\rho_i + \delta)
  \end{align*}
  and area $A^\delta = \pi$.
\end{corollary}

\section{Relation to orthogonal circle patterns}
\label{sec:circle_patterns}

In this section we give a detailed description of the relation of 
orthogonal ring patterns and orthogonal circle patterns.
It turns out that orthogonal circle patterns can be considered as a special
case of ring patterns with constant ring area~$A_0 = 0$.

To formulate the limit we need to review some properties of orthogonal circle patterns.
Two orthogonally intersecting circles in an orthogonal circle pattern create
a cyclic right angled kite (see Fig.~\ref{fig:kite_deformation} left and right).
The angle~$\varphi^\circ_{\ij}$ at a vertex $v_i$ in a kite on the edge $(v_i, v_j)$
of an orthogonal circle pattern with radii~$r^\circ_i = e^\rho_i$ is given by:
\begin{equation}
  \label{eq:cp_angle}
  \begin{aligned}
    \varphi^\circ_{\ij} 
    &= 2 \arctan(\frac{r^\circ_j}{r^\circ_i}) = 2 \arctan(e^{\rho_j-\rho_i})\\
    &= \pi - 2 \arctan(e^{\rho_i -\rho_j})
  \end{aligned}
\end{equation}
In case of circle patterns the $\rho$-radii are called \emph{logarithmic radii}.
Logarithmic radii of an immersed orthogonal circle pattern are governed by the same equation (cf.~\cite{schramm1997,
bobenko_springborn_variational}) as the $\rho$-radii of ring patterns (see
Thm.~\ref{thm:ring_pattern}).

Furthermore, for each orthogonal circle pattern~$\mathcal{C}$ with logarithmic
radii $\rho_i$ there exists a dual pattern~$\mathcal{C}^*$ with radii~$e^{-\rho_i}$. 
The angles of the dual pattern are given by
\[
  (\varphi^\circ_{\ij})^* = 2 \arctan(\frac{r^*_j}{r^*_i}) 
  = 2 \arctan(e^{-\rho_j+\rho_i}) = \pi - \varphi^\circ_{\ij}\,.
\]
Note that the angles at interior vertices still sum up to $2\pi$, but the
angles at the boundary vertices change as shown in
Fig.~\ref{fig:reciprocal_radii}.

\begin{figure}[tb]
  \centering
  \includegraphics[height=4cm]{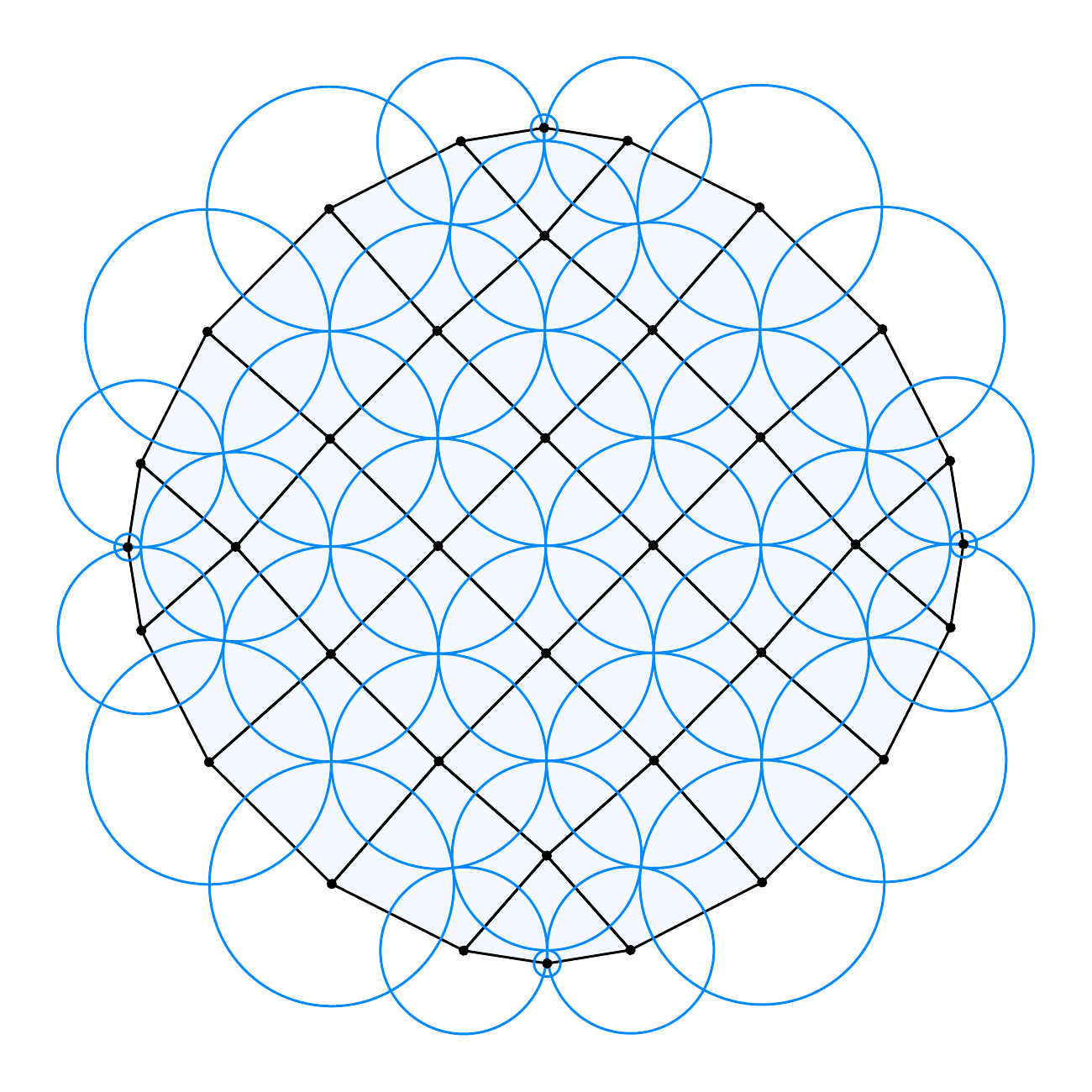}
  \hspace{1cm}
  \includegraphics[height=4cm]{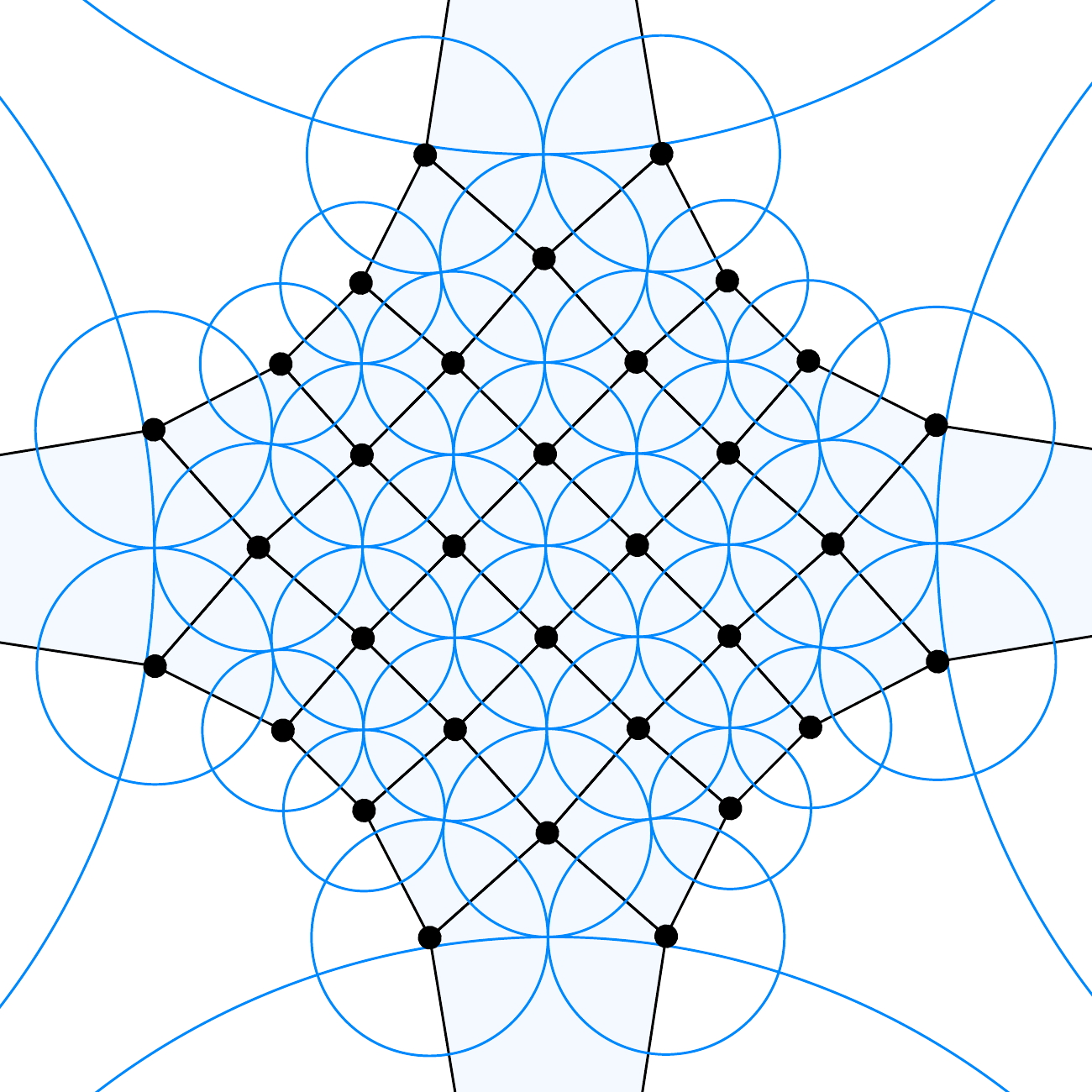}
  \caption{
    An orthogonal circle pattern and its dual. 
    The boundary angles the dual pattern are $2\pi - \varphi_{\ij}$
    resp.\ $\pi - \varphi_{\ij}$ depending on whether the degree of the
    boundary vertex is $3$ or $2$.
  }
  \label{fig:reciprocal_radii}
\end{figure}

Now let us go back to the one parameter family~$\mathcal{R}^\delta$ of ring
patterns defined in Cor.~\ref{cor:pattern_deformation}.
To avoid that the radii go to infinity as $\delta \to \pm \infty$ we scale the 
entire pattern by $2 e^{-|\delta|}$. 
So the radii of the one parameter family of ring patterns are:
\[
  r_{m,n}^\delta = 2 e^{-|\delta|} \sinh(\rho_{m,n} + \delta)
  \quad\text{and}\quad
  R_{m,n}^\delta = 2 e^{-|\delta|} \cosh(\rho_{m,n} + \delta).
\]
In the limit $\delta \to \pm \infty$ the areas of the rings tend to zero and
for the radii we have: 

\begin{align*}
  \lim_{\delta\to\pm\infty} r^\delta_i
  &= \lim_{\delta\to\pm\infty} 2 e^{-|\delta|} \frac{1}{2}(e^{\rho_i + \delta} - e^{- \rho_i - \delta})
  =\pm e^{\pm\rho_i}\,,\\
  \lim_{\delta\to\pm\infty} R^\delta_i
  &= \lim_{\delta\to\pm\infty} 2 e^{-|\delta|} \frac{1}{2}(e^{\rho_i + \delta} + e^{- \rho_i - \delta})
  = e^{\pm\rho_i}.
\end{align*}

{\textbf{Remark 3.0.}} (Limits on compact subsets).
If the ring pattern $\mathcal{R}$ is infinite we  consider the limits $\delta\to\pm\infty$ 
of the family $\mathcal{R}^\delta$ on any  compact subset $G_0\subset G$ satisfying the same conditions as $G$, i.e. $G_0$ and $G_0^*$ are simply connected.

{\textbf{Limit} $\delta \to +\infty$.} 
For $\delta > -\min_{v_i \in G_0} \rho_i$
we have $\rho_i^\delta = \rho_i + \delta > 0$ for all $v_i \in G_0$.  So
considering the limit as $\delta \to \infty$ all $\rho^\delta_i$ will be
positive and the angles of the circle pattern~$\mathcal{C}$
(equation~\eqref{eq:cp_angle}) are exactly those of the ring pattern
$\mathcal{R^\delta}$ given in Lemma~\ref{lem:kite_angle}.  Furthermore, for
$\delta\to\infty$, we obtain rings with area~$0$ since the outer and
inner radii both converge to~$e^{\rho_i}$. The neighboring circles intersect
orthogonally because inner and outer circles of the orthogonal ring pattern are
intersecting orthogonally in the entire one parameter family.  
%As the equation
%for the existence of orthogonal ring patterns and for the existence of
%orthogonal circle patterns coincide 
The limit circles form a Schramm type
orthogonal circle pattern.

{\textbf{Limit} $\delta \to -\infty$.}
For $\delta < -\max_{v_i \in G_0} \rho_i$ all $\rho_i^\delta =
\rho_i + \delta < 0$. By Lemma~\ref{lem:kite_angle}
the angles of the ring pattern for negative~$\rho_i$ are given by 
\[
  \varphi_{\ij} 
  = - 2\arctan (e^{\rho_i - \rho_j}) 
  = - \pi + \arctan(e^{(-\rho_i) - (-\rho_j)}) 
\] 
and correspond to the angles of the dual pattern~$\mathcal{C}^*$ with opposite
orientation. As equation~\eqref{eq:closure} is satisfied for all $\delta$, we obtain
the dual orthogonal circle pattern~$\mathcal{C}^*$ (with opposite orientation)
in the limit. 

\begin{corollary}
  Let~$\mathcal{R}^\delta$ be a one parameter family of orthogonal ring
  patterns with $\rho^\delta_i = \rho_i + \delta$ for $\rho_i \in \R$ as
  described in Cor.~\ref{cor:pattern_deformation}.  Then for $\delta \to
  +\infty$ we obtain an orthogonal circle pattern~$\mathcal{C}$ with logarithmic
  radii~$\rho_i$ and for $\delta \to -\infty$ we obtain the dual
  circle pattern~$\mathcal{C}^*$ with logarithmic radii~$-\rho_i$.
\end{corollary}

Here the limits are understood in the sense of Remark 3.0.

For a better understanding of the deformation, 
%let us have a look at 
the one
parameter family of cyclic quadrilaterals associated to a single edge~$(v_i,
v_j)$ 
is
shown in Fig.~\ref{fig:kite_deformation}: Assume that $\rho_i$ and
$\rho_j$ are both positive and $\rho_i < \rho_j$.  Then the deformation starts
with an embedded cyclic quadrilateral (center right). For
$\delta \to \infty$ we obtain two orthogonally intersecting circles with radii
$e^{\rho_i}$ and $e^{\rho_j}$ that form a kite (bottom right).  When $\delta \searrow
-\rho_i$ one of the edges at $v_i$ shrinks to a point and reverses its direction as
$\rho_i + \delta$ changes its sign from ${+}$ to ${-}$. If $-\rho_j < \delta <
-\rho_i$ then $r_i^\delta < 0$ and we obtain a non-embedded quadrilateral (top center).
Again as $\delta \searrow -\rho_j$ one edge at $v_j$ shrinks to a point and
changes its direction as $\rho_j + \delta$ changes sign (center left) and we
obtain an embedded quadrilateral with negative orientation.
For $\delta \to -\infty$ the areas of the rings go to zero and we obtain two
orthogonally intersecting circles with radii $e^{-\rho_i}$ and $e^{-\rho_j}$
(bottom left).
%In the limit for $\rho \to \pm \infty$ we can consider the deformation of a
%point ($\rho = -\infty$) to a line ($\rho = \infty$).

\begin{figure}[bt]
  \centering
  \includegraphics[width=1.0\linewidth]{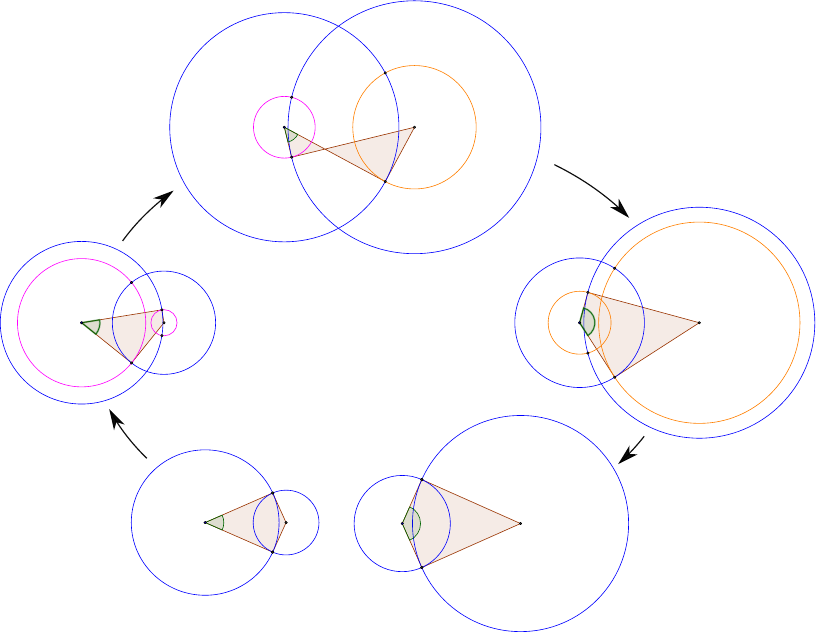}
  \caption{
    Deformation of a cyclic quadrilateral defined by two orthogonally
    intersecting rings. The bottom left and bottom right show the limits of
    the ring pattern as the area of the ring goes to zero. Positive
    radii are indicated by orange, negative radii (i.e., negative~$\rho$) are
    indicated by pink circles. The angle associated with the left vertex is
    shown in green.
  }%
  \label{fig:kite_deformation}
\end{figure}

\section{Doyle spiral, Erf, and $z^\alpha$ ring patterns}
\label{sec:examples}

In this section we will have a look at some known orthogonal circle patterns and
consider their ring pattern analogs and deformations.

\subsection{Doyle spirals}
\label{sec:doyle_spirals}

Doyle spirals for the square lattice have been constructed by
Schramm~\cite{schramm1997}.  For $x + iy \in \C
\setminus\{0\}$ Schramm defines radii by~$r_{ m,n } = |e^{(x+iy)(m + in)}|$.
Taking the logarithm we obtain the logarithmic radii $\rho_{m,n} = mx - ny$.
We will take these radii as a definition of the Doyle spiral ring pattern.

\begin{proposition}[Doyle spiral ring pattern]
  \label{prop:doyle_rp}
  Let $x+iy \in \C \setminus \{0\}$ be a complex number.
  The \emph{Doyle spiral ring pattern} is given by the $\rho$-radii
  $\rho_{ m,n } = mx - ny$ for $(m,n) \in \Z^2$.
\end{proposition}

Let us consider the generic case $\frac{x}{y} \not\in {\mathbb Q}$ when the $\rho$-radii do not vanish. By Lemma~\ref{lem:kite_angle} the angles of the cyclic quadrilaterals at the edges
are given by 
\begin{align*}
  \varphi_{(m,n),(m+1,n)} &= 
  \begin{cases}
    \pi - 2 \arctan(e^x) & \text{if $\rho_{m,n} > 0$}\\
    - 2 \arctan(e^{x}) & \text{if $\rho_{m,n} < 0$}
  \end{cases}
  \quad\text{and}\\
  \varphi_{(m,n),(m,n+1)} &= 
  \begin{cases}
    \pi - 2 \arctan(e^{-y}) & \text{if $\rho_{m,n} > 0$}\\
    - 2 \arctan(e^{-y}) & \text{if $\rho_{m,n} < 0$}
  \end{cases}
\end{align*}

Looking closer at the signs of the $\rho$-radii we observe that
\begin{align*}
  \rho_{m,n} > 0 &\Leftrightarrow mx > ny\, &\text{and}&&
  %\rho_{m,n} = 0 &\Leftrightarrow (m,n) = (0,0)\,\text{, and}\\
  \rho_{m,n} < 0 &\Leftrightarrow mx < ny. 
\end{align*}
So the signs of the $\rho$-radii change across the line $\{ (m,n) \in \Z^2
\,|\, mx = ny\}$ and hence does the orientation of the flowers. If we restrict
to the parts $\{ (m,n) \in \Z^2 \,|\, mx > ny\}$ (resp.\ $\{ (m,n) \in \Z^2
\,|\, mx < ny\}$) we see that the angles are constant for all horizontal edges
$(m,n)(m+1,n)$ and all vertical edges $(m,n)(m,n+1)$.  Thus we can define a
Doyle spiral ring pattern by two angles~$\alpha$ and $\beta$, one for
the horizontal and one for the vertical direction.
This is the characteristic property for the Doyle spiral circle pattern.

Consider the one parameter family~$\mathcal{R}^\delta$ of orthogonal %circle
ring patterns as described by Cor.~\ref{cor:pattern_deformation}.  The angles
along the horizontal and vertical edges stay constant in the two halfspaces.
As in the general case discussed in the previous section, all $\rho$'s become
positive for $\delta\to +\infty$ (resp.\ negative for $\delta\to-\infty$), see Remark 3.0, and we
obtain a Doyle spiral and its dual as constructed by Schramm (see
Fig.~\ref{fig:doyle_spiral_deformation}).

\begin{figure}[tb]
  \centering
  \includegraphics[width=.48\linewidth]{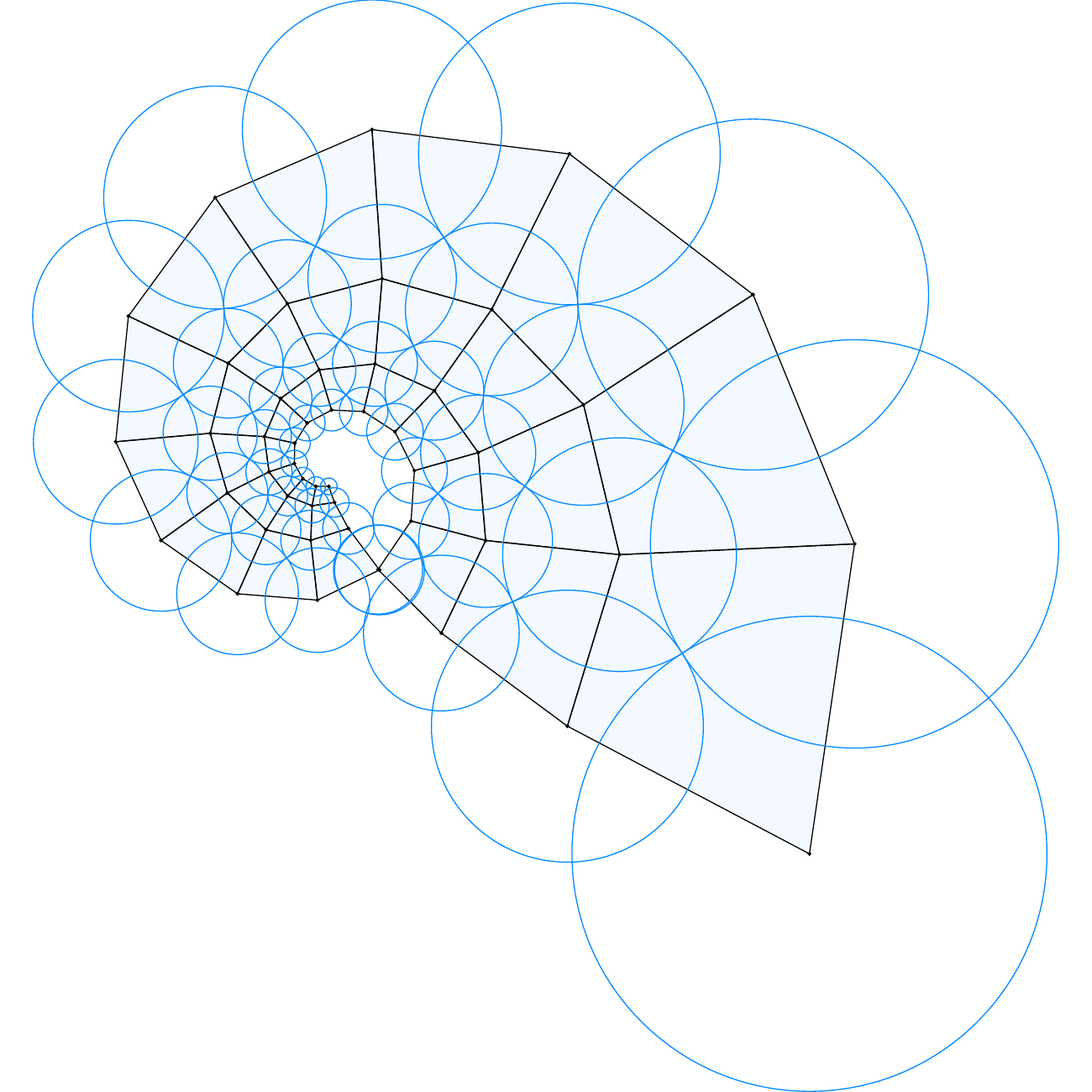}
  \includegraphics[width=.48\linewidth]{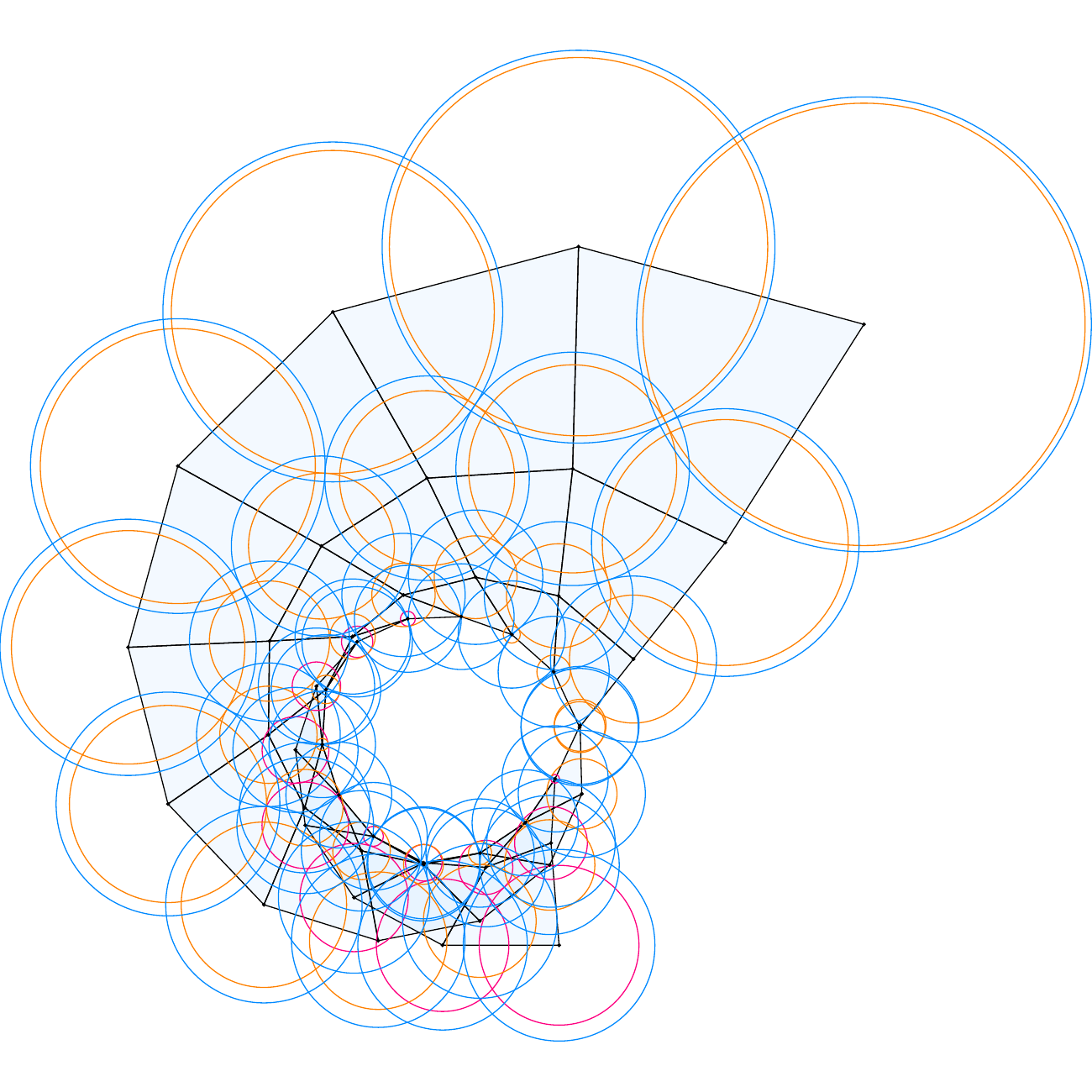}
  \includegraphics[width=.48\linewidth]{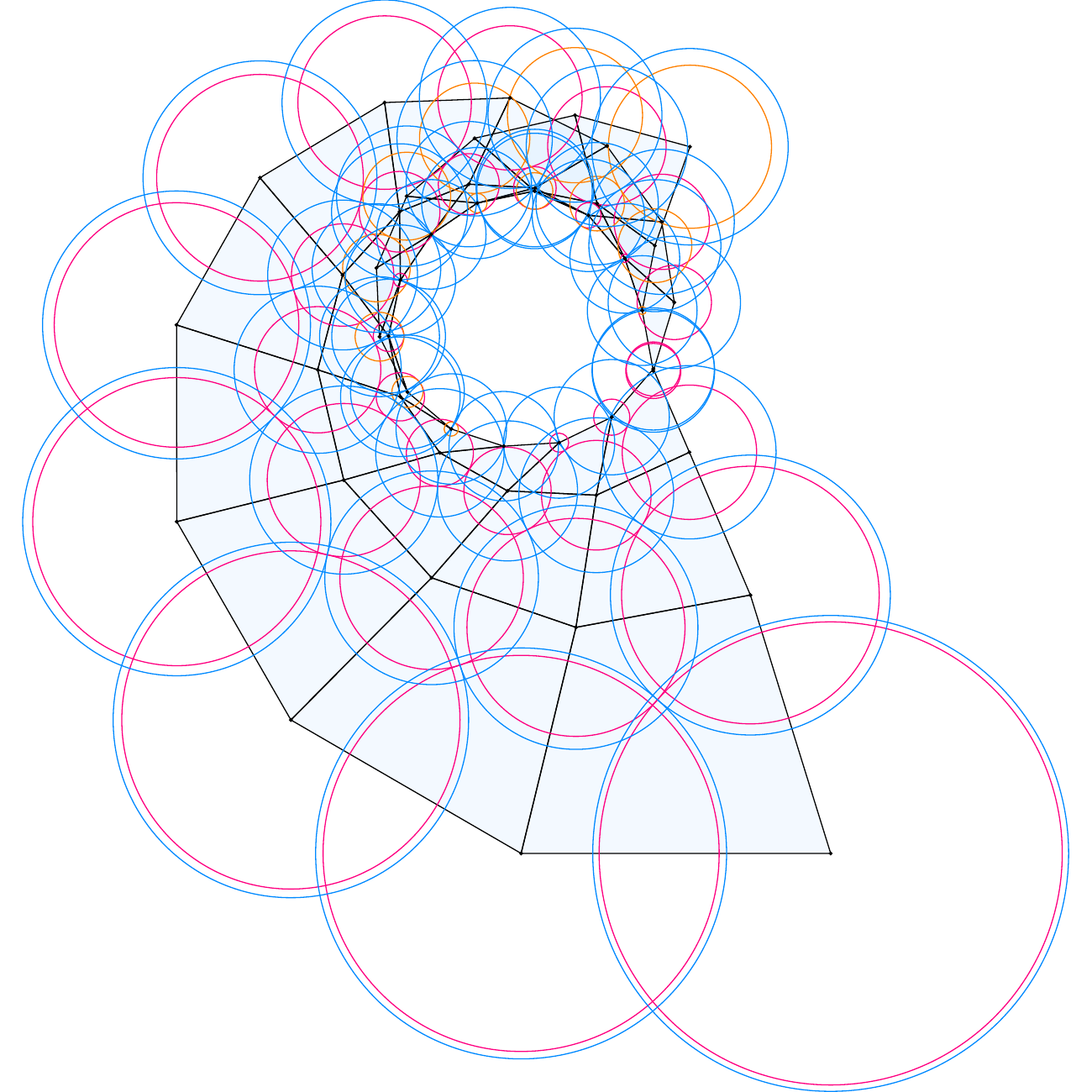}
  \includegraphics[width=.48\linewidth]{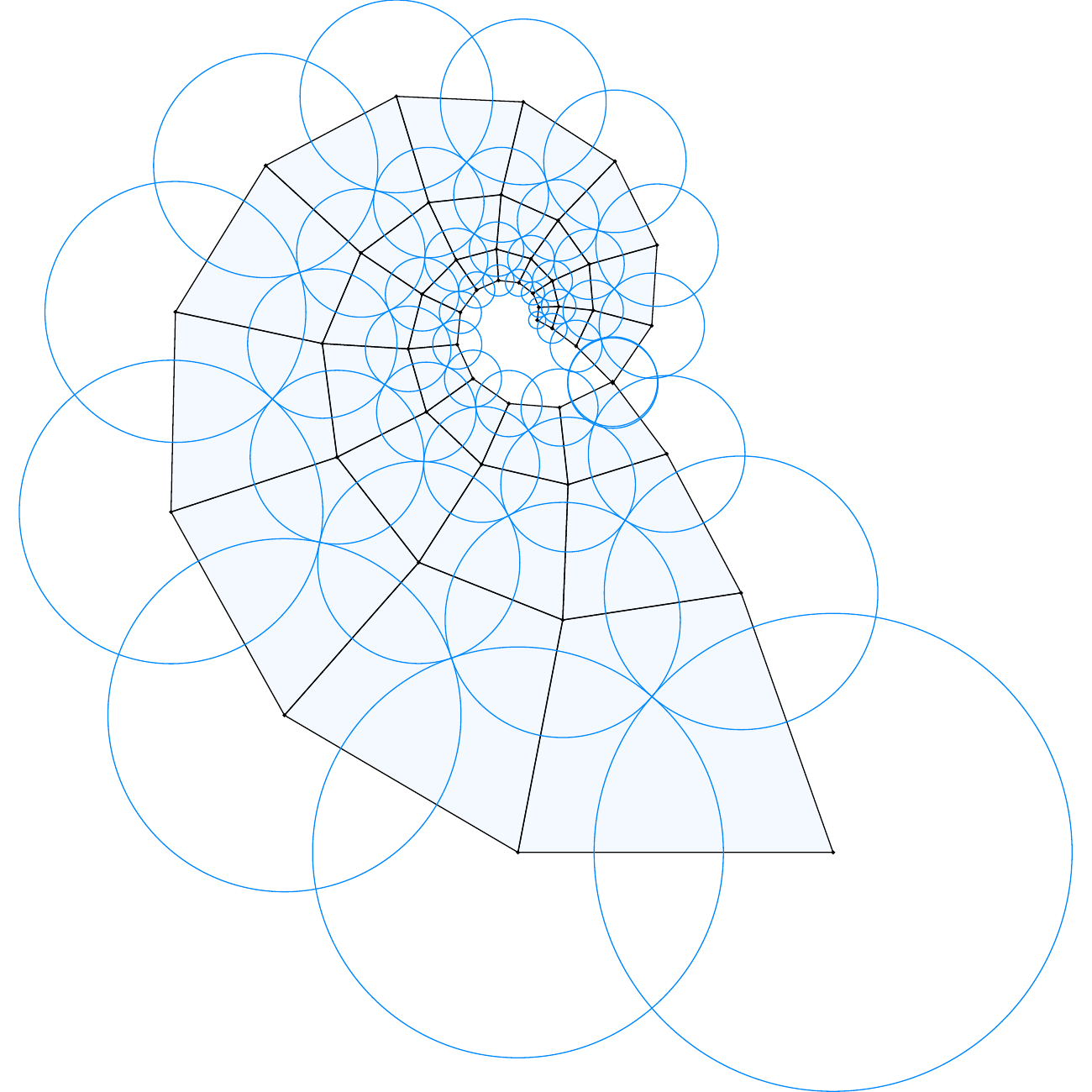}
  \caption{
    Deformation of an orthogonal circle pattern (top left) into its dual
    (bottom right) through a one parameter family of ring patterns (top
    right and bottom left). We see how the orientation of the quadrilaterals
    flips during the deformation. The innermost vertex in the top left circle
    patterns becomes the outermost vertex in the bottom right circle pattern.
  }%
  \label{fig:doyle_spiral_deformation}
\end{figure}

\subsection{Erf pattern}
\label{sec:erf}

For analogs to Schramm's $\sqrt{i}$-Erf pattern let us have a look at the
corresponding radius function given in~\cite{schramm1997} $ r_{m,n} =
e^{a m n}$ for $(m,n) \in \Z^2$ and $a \in \R$, $a > 0$.  Taking the
logarithm we obtain $\rho_{m, n} = a m n$.  As in case of the Doyle spiral we
will use this function to define the corresponding ring patterns.

\begin{proposition}[Erf ring pattern]
  \label{prop:erf_rp}
  Let $a \in \R, a > 0$.
  The \emph{Erf ring pattern} is given by the $\rho$-radii
  $\rho_{ m,n } = a m n $ for $(m,n) \in \Z^2$.
\end{proposition}

The angles in the pattern are given by
\begin{align*}
  \varphi_{(m,n),(m+1,n)} &= 
  \begin{cases}
    \pi - 2 \arctan(e^{-an}) & \text{if $\rho_{m,n} > 0$}\\
    - 2 \arctan(e^{-an}) & \text{if $\rho_{m,n} < 0$}
  \end{cases}
  \quad\text{and}\\
  \varphi_{(m,n),(m,n+1)} &= 
  \begin{cases}
    \pi - 2 \arctan(e^{-am}) & \text{if $\rho_{m,n} > 0$}\\
    - 2 \arctan(e^{-am}) & \text{if $\rho_{m,n} < 0$}
  \end{cases}
\end{align*}
As $\rho_{m,n} = amn$ the $\rho$-radii change signs at the coordinate axes.
In the four quadrants, the angles along the horizontal and the vertical
parameter lines are constant. All the rings on the coordinate axes are congruent: the radii of their outer circles are equal to $R=\cosh 0=1$, and their inner circles degenerate to their centers.

If we consider the one parameter family of ring patterns defined in
Cor.~\ref{cor:pattern_deformation} we see that in the limit $\delta\to +\infty$
we obtain the $\sqrt{i}$-SG Erf circle patterns constructed by Schramm.
For $\delta\to-\infty$ we obtain a pattern with $\rho^*_{m,n} = -amn$. This is 
the same pattern as for $a$ since $\rho^*_{m,n} = \rho_{-m,n}$.

\begin{figure}[tb]
  \centering
  \includegraphics[width=\linewidth]{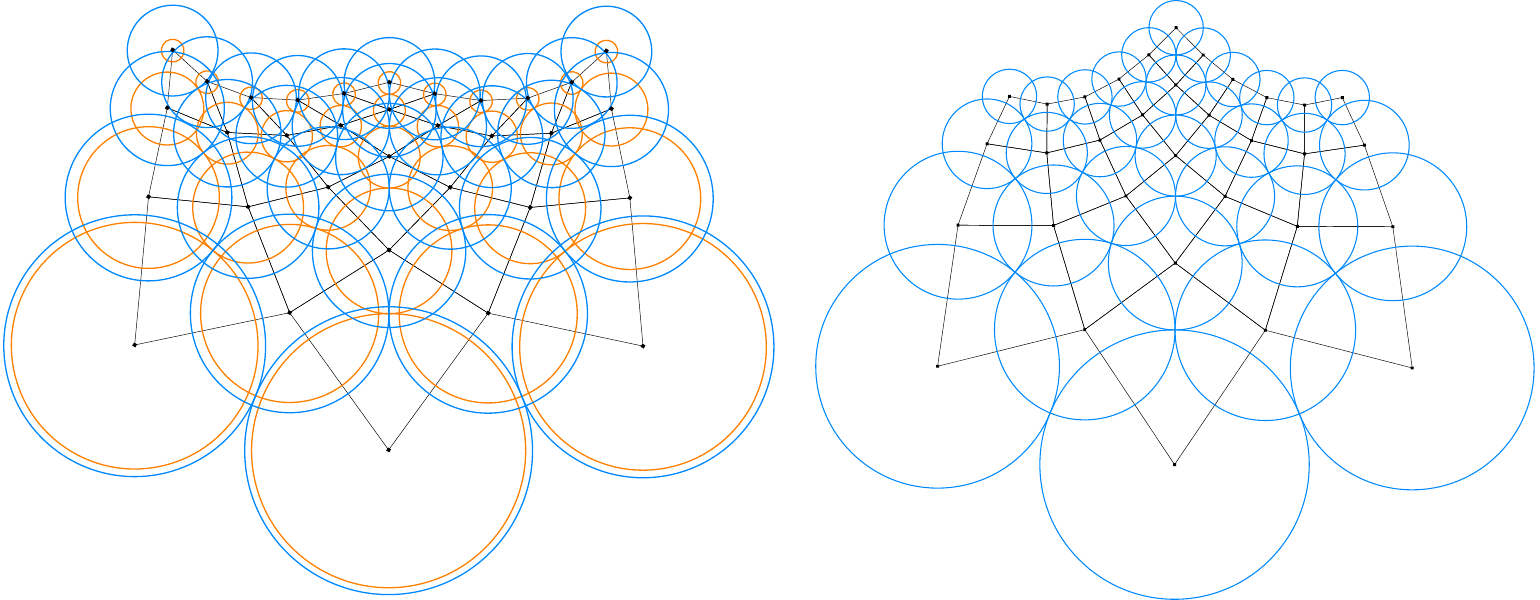}
  \caption{An Erf ring pattern (left) and the corresponding limit 
  circle pattern.}
  \label{fig:erf_pattern}
\end{figure}

\subsection{$z^\alpha$ and logarithm patterns}
\label{ssec:zalpha}
\begingroup
\newcommand{\mm}{m}
\newcommand{\nn}{n}

In \cite{agafonov_bobenko_zalpha} the authors defined an orthogonal circle
pattern~$\mathcal{C}(z^\alpha)$ as a discretization of the complex map~$z
\mapsto z^\alpha$ for $\alpha \in (0, 2)$.  The radius function of the circle
pattern is given by the
following identities (cf. \cite[Thm.~3,
equation~(10, 11)]{agafonov_bobenko_zalpha}) on a subset of~$\Z^2$ given by
$V = \{(\mm,\nn) \,|\, \mm \geq |\nn|\}$:
\begin{multline*}
  r_{\mm,\nn} r_{\mm+1,\nn}(-2\nn-\alpha) + r_{\mm+1,\nn} r_{\mm+1,\nn+1} (2(\mm+1)-\alpha)\\
  +r_{\mm+1,\nn+1}r_{\mm,\nn+1}(2(\nn+1)-\alpha) + r_{\mm,\nn+1}r_{\mm,\nn}(-2\mm-\alpha) = 0
\end{multline*}
for $V \cup \{(-\mm,\mm-1) \,|\, \mm \in \N \}$ and
\begin{multline*}
  (\mm + \nn)(r_{\mm,\nn}^2 - r_{\mm+1,\nn}r_{\mm,\nn-1})(r_{\mm,\nn+1} + r_{\mm+1,\nn}) \\
  + (\nn - \mm)(r_{\mm,\nn}^2 - r_{\mm,\nn+1} r_{\mm+1,\nn})(r_{\mm+1,\nn}+r_{\mm,\nn-1}) = 0
\end{multline*}
for interior vertices $V \setminus \{(\pm \mm, \mm) \,|\, \mm \in \N\}$ 
with initial condition~$r_{0,0} =1 $ and $r_{1,0}=r_{0,1} = \tan \frac{\alpha\pi}{4}$.

It is known that the dual pattern of $z^\alpha$ is given by $z^{2-\alpha}$,
e.g., the dual circle pattern of $\mathcal{C}(z^{2/3})$ is
$\mathcal{C}(z^{4/3}) = (\mathcal{C}(z^{2/3}))^*$ shown in
Fig.~\ref{fig:z23_43} (top left and bottom right).  Based on the logarithmic
radii of these patterns we construct a one parameter family of ring patterns
that interpolates between the two patterns.

\begin{figure}[t]
  \centering
  \includegraphics[width=.49\linewidth]{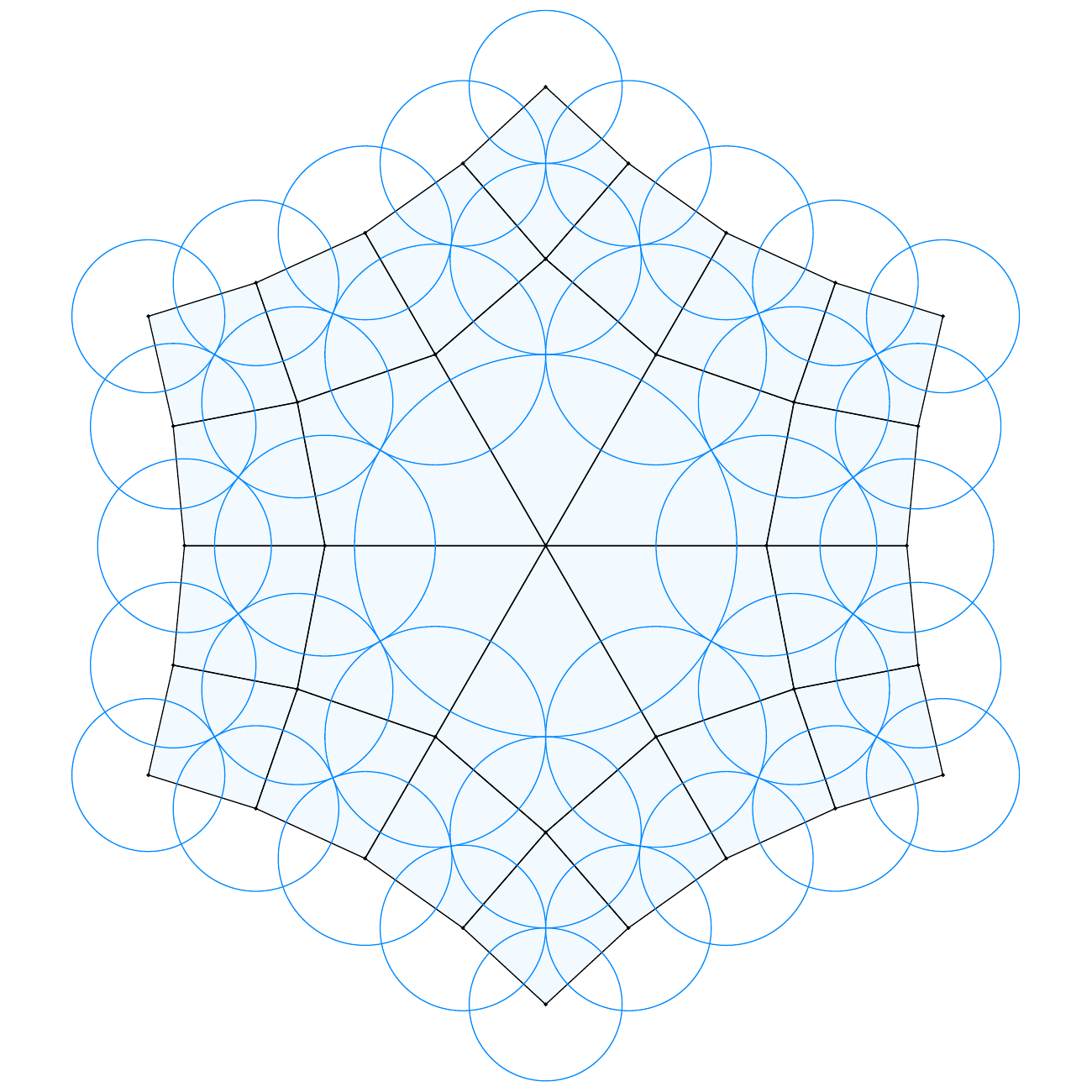}
  \includegraphics[width=.49\linewidth]{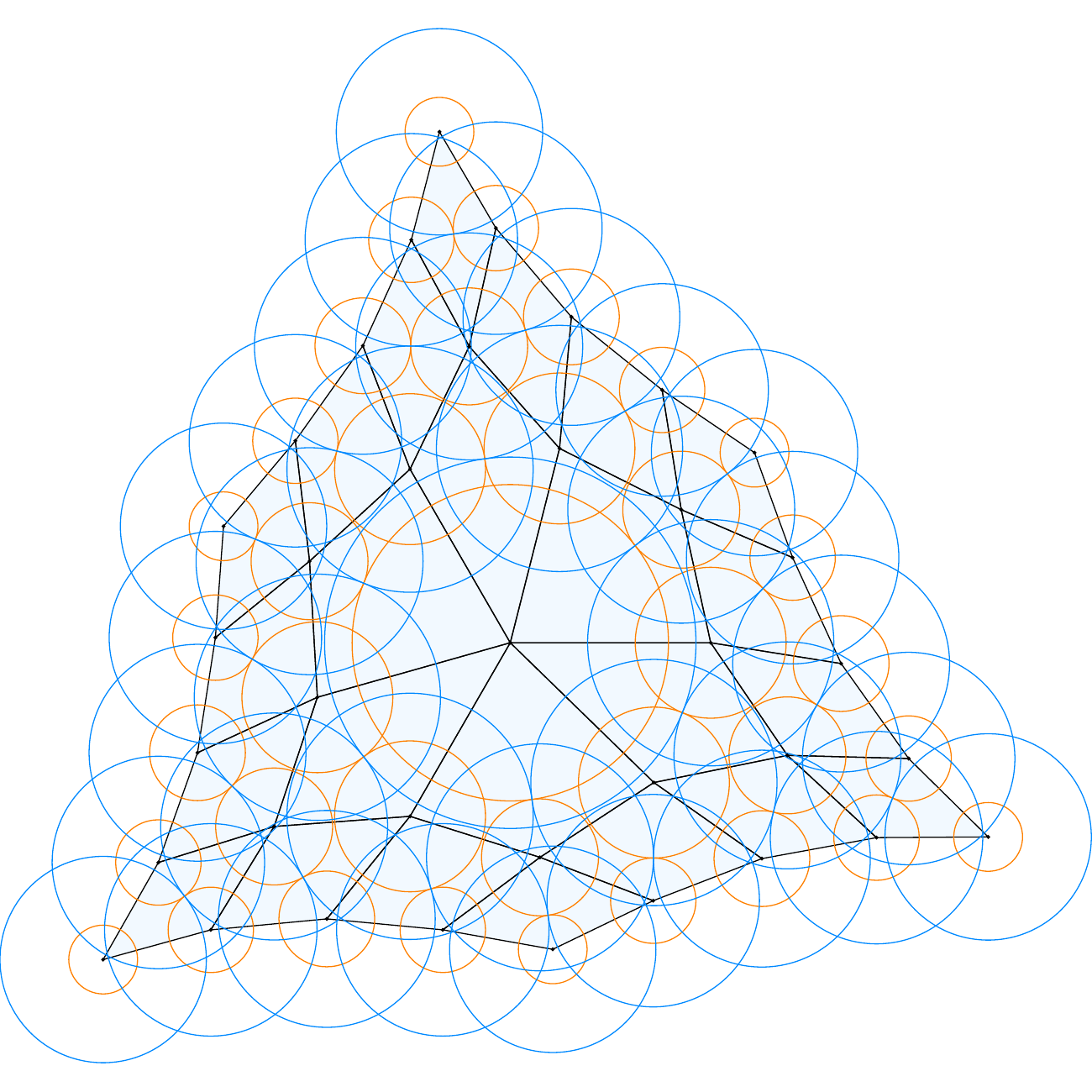}

  \includegraphics[width=.49\linewidth]{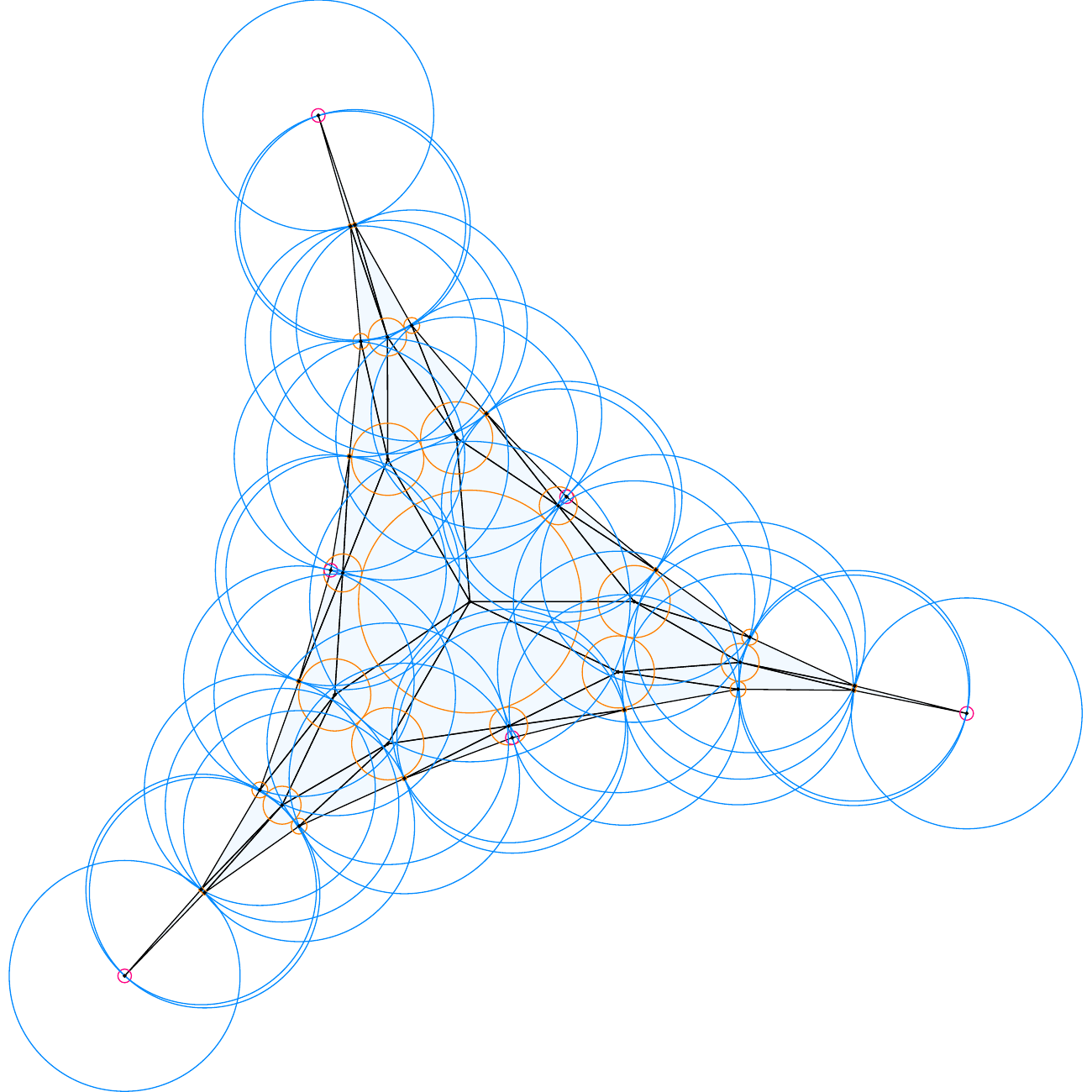}
  \includegraphics[width=.49\linewidth]{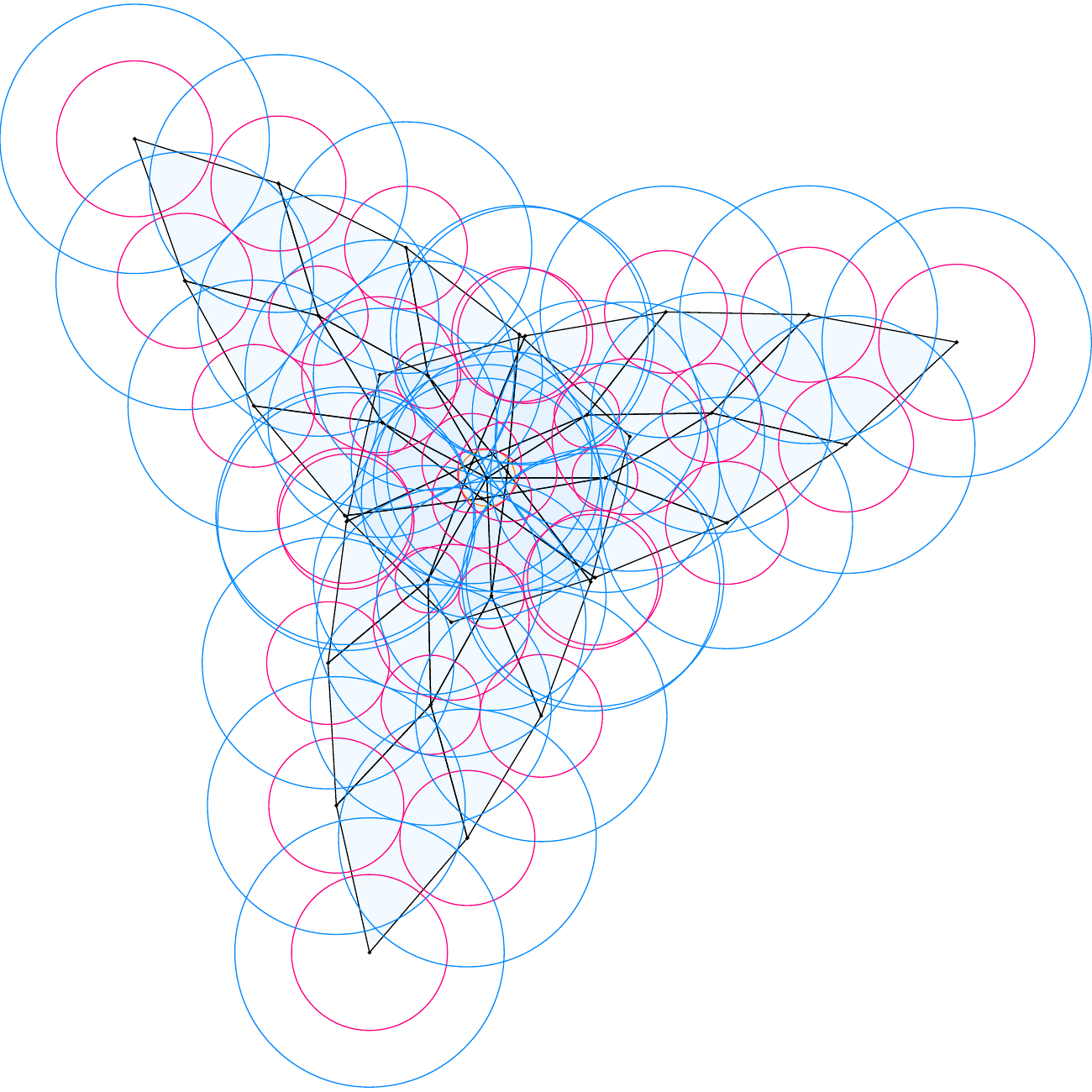}

  \includegraphics[width=.49\linewidth]{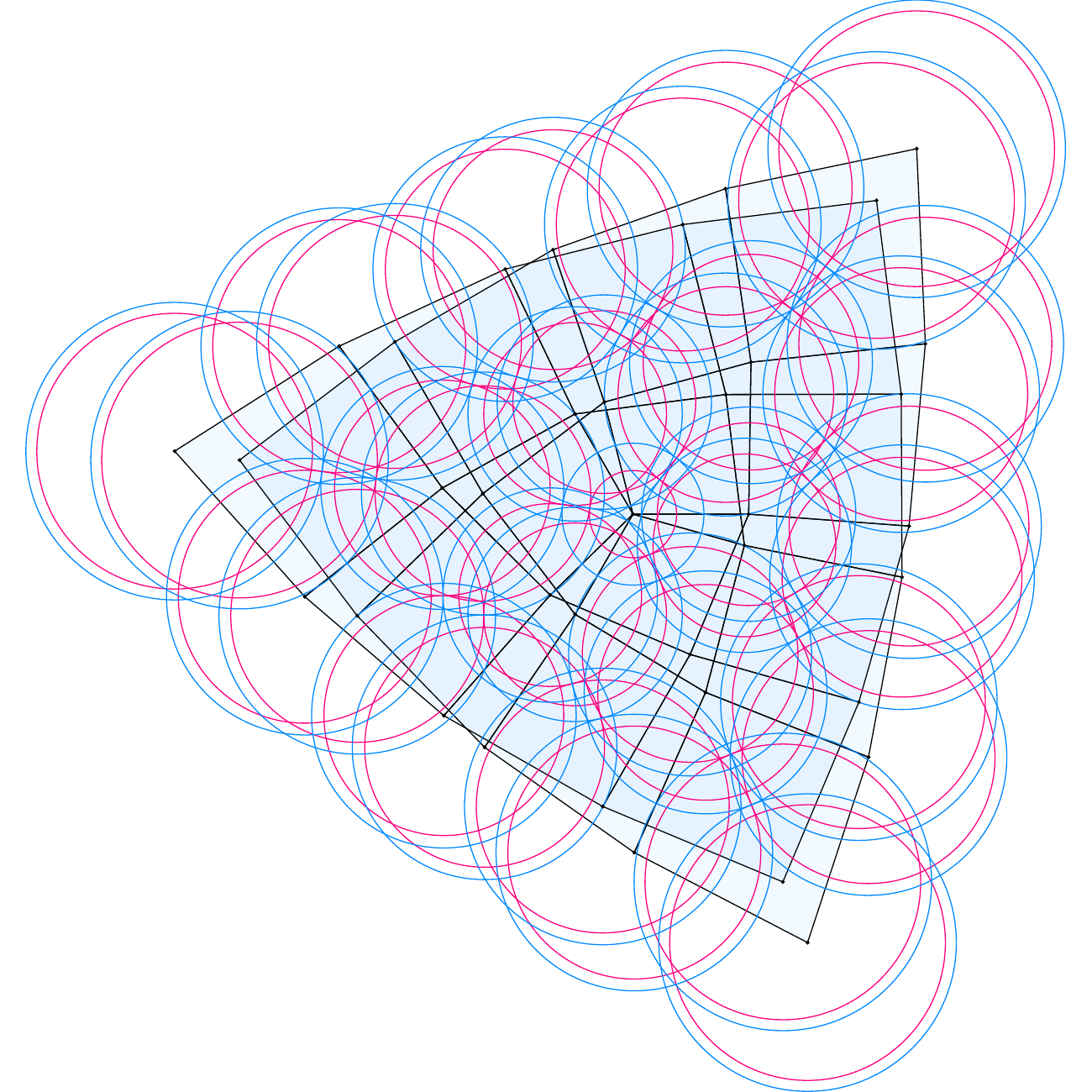}
  \includegraphics[width=.49\linewidth]{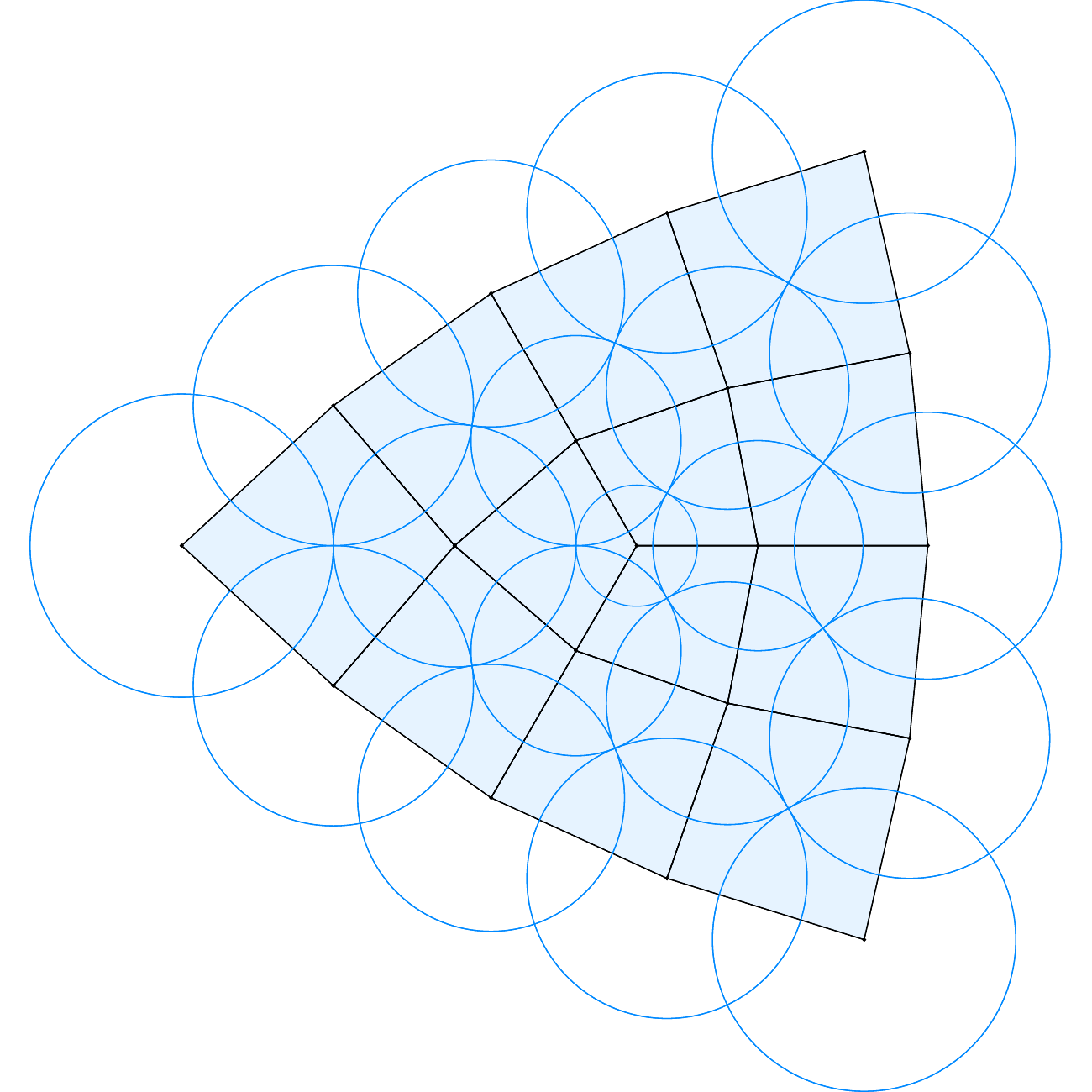}
  \caption{One parameter family of orthogonal ring patterns interpolating
    between the orthogonal circle pattern for $z \mapsto z^{2/3}$ (top left) and the
    dual pattern for $z^{4/3}$ (bottom right)}
  \label{fig:z23_43}
\end{figure}

An orthogonal circle pattern for~$z^2$ can be defined by considering a special
limit for $\alpha \to 2$. The radii of the~$z^2$ pattern are defined
in~\cite[Sect.~5]{agafonov_bobenko_zalpha}.  The dual of $z^2$ is the logarithm
map~$\log z$.  In each of the corresponding orthogonal circle patterns, one of
the circles degenerates.  In case of $z^2$ one of the circles has radius~$0$,
i.e., the circle degenerates to a point and the logarithmic radius is negative
infinity.  Consequently, one of the circles in the $\log z$ pattern has radius
infinity, i.e., the circle degenerates to a line and the logarithmic radius is
positive infinity.  We illustrate the one parameter deformation of~$z^2$ to
$\log(z)$ in Fig.~\ref{fig:z2_log}.

\begin{figure}[tb]
  \centering
  \includegraphics[width=.48\linewidth]{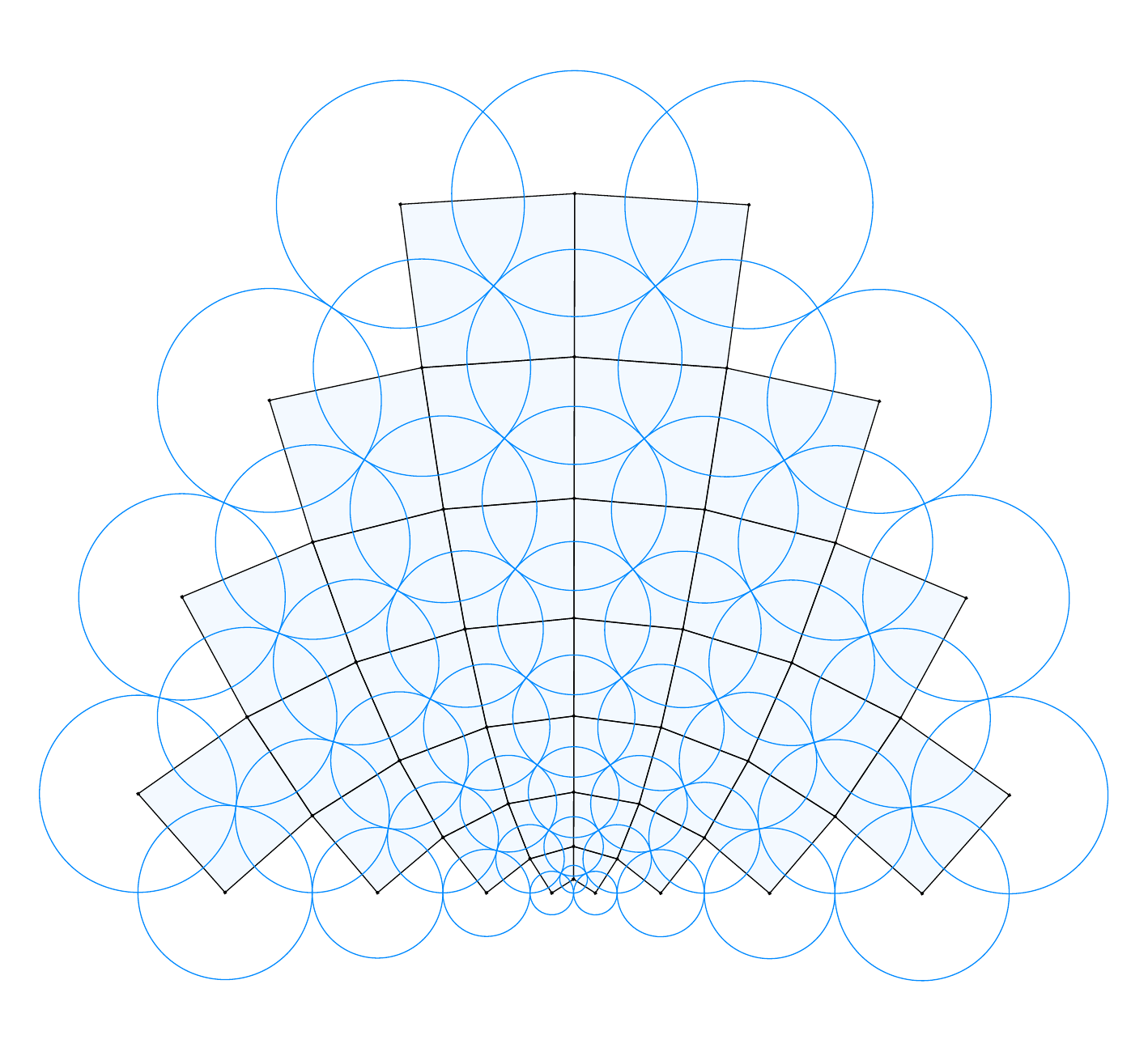}
  \includegraphics[width=.48\linewidth]{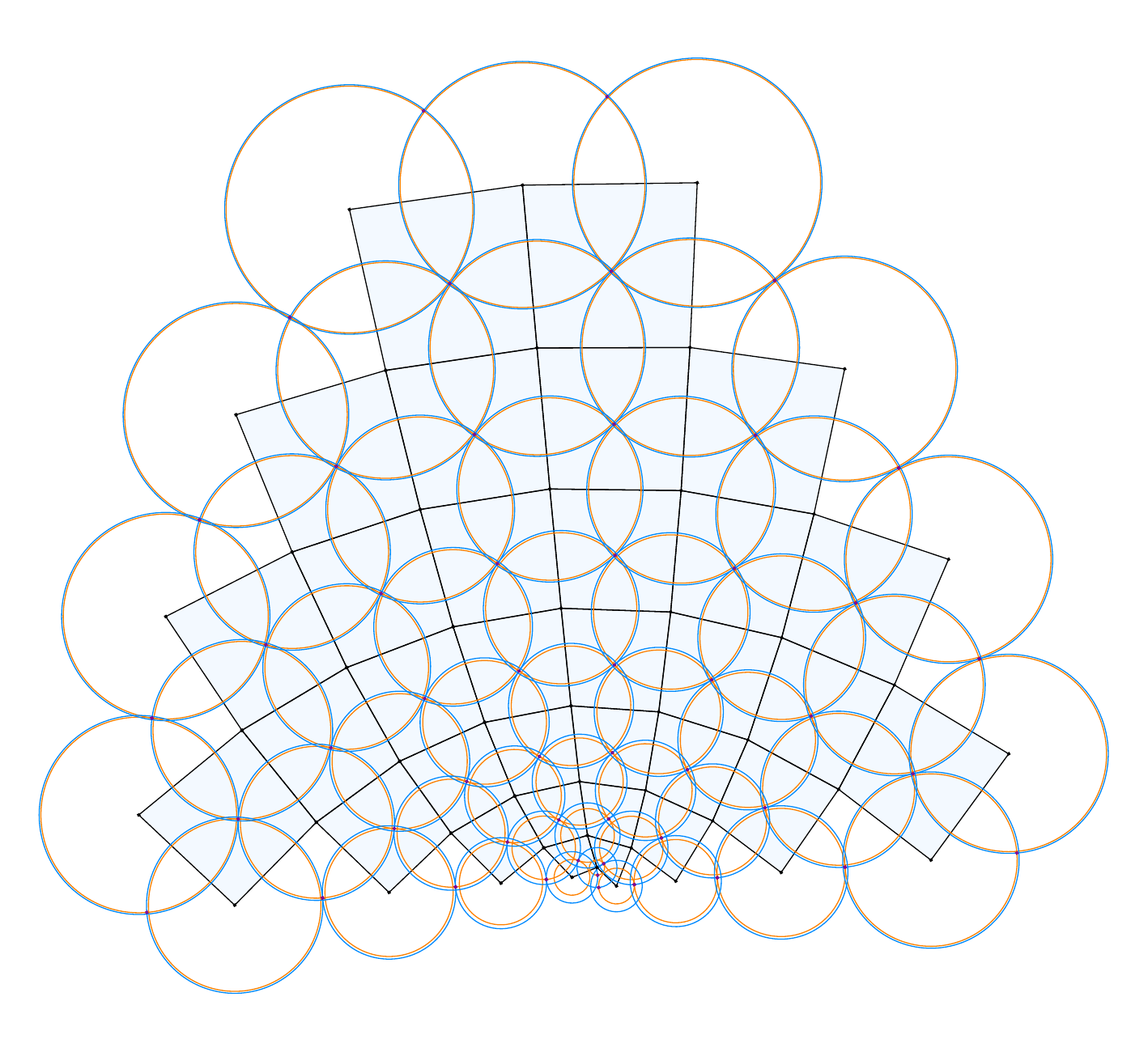}

  \includegraphics[width=.48\linewidth]{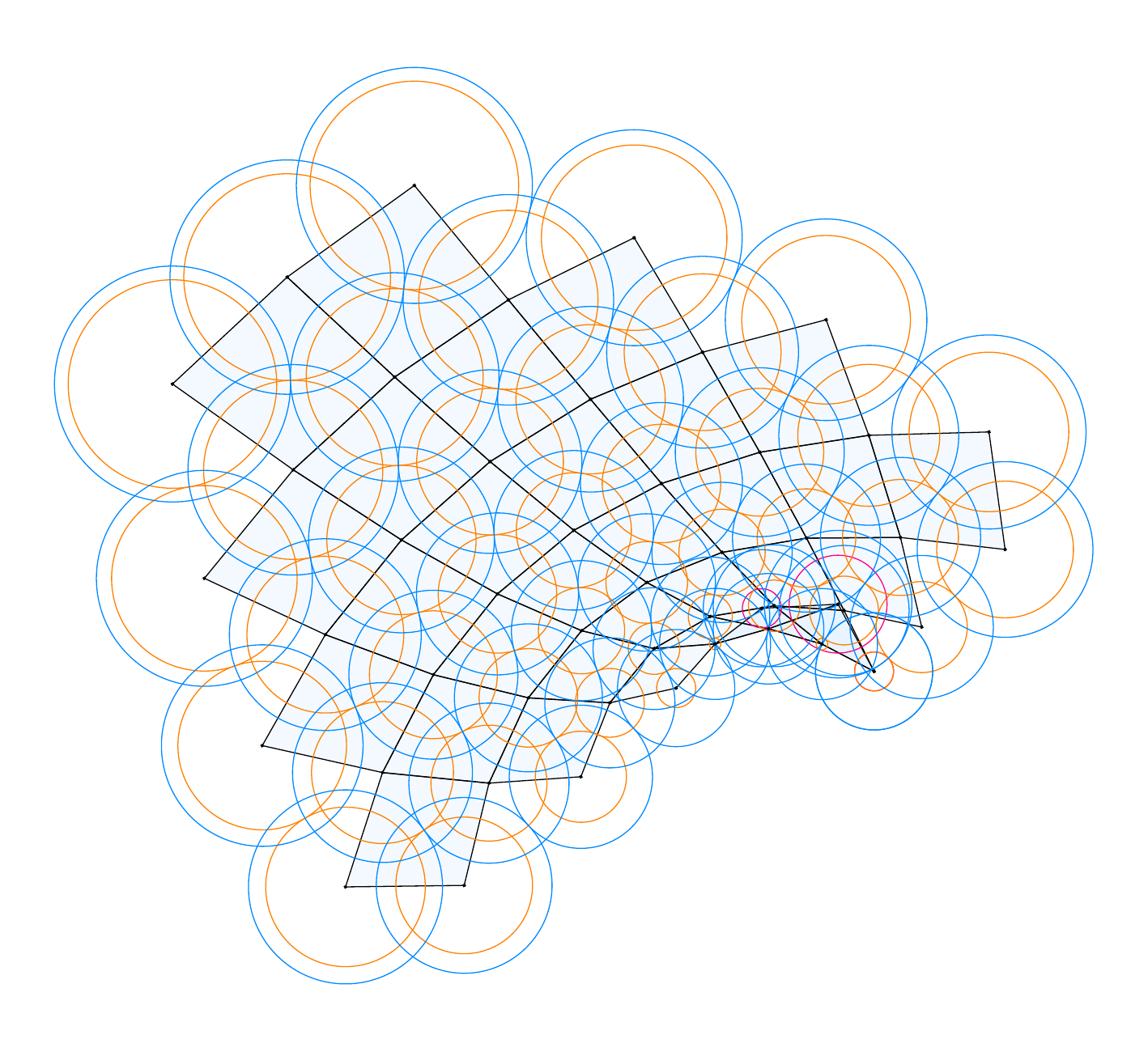}
  \includegraphics[width=.48\linewidth]{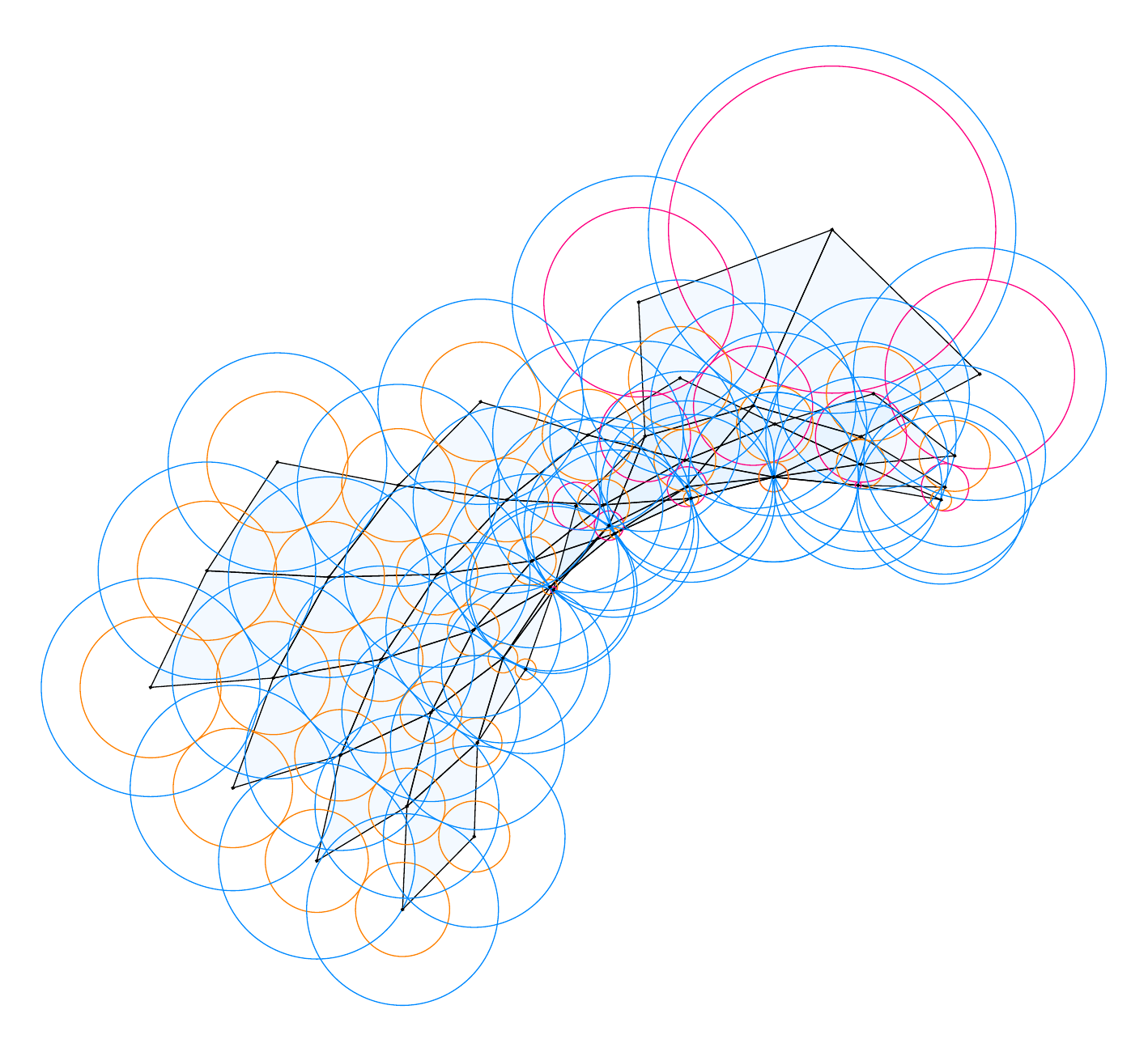}

  \includegraphics[width=.48\linewidth]{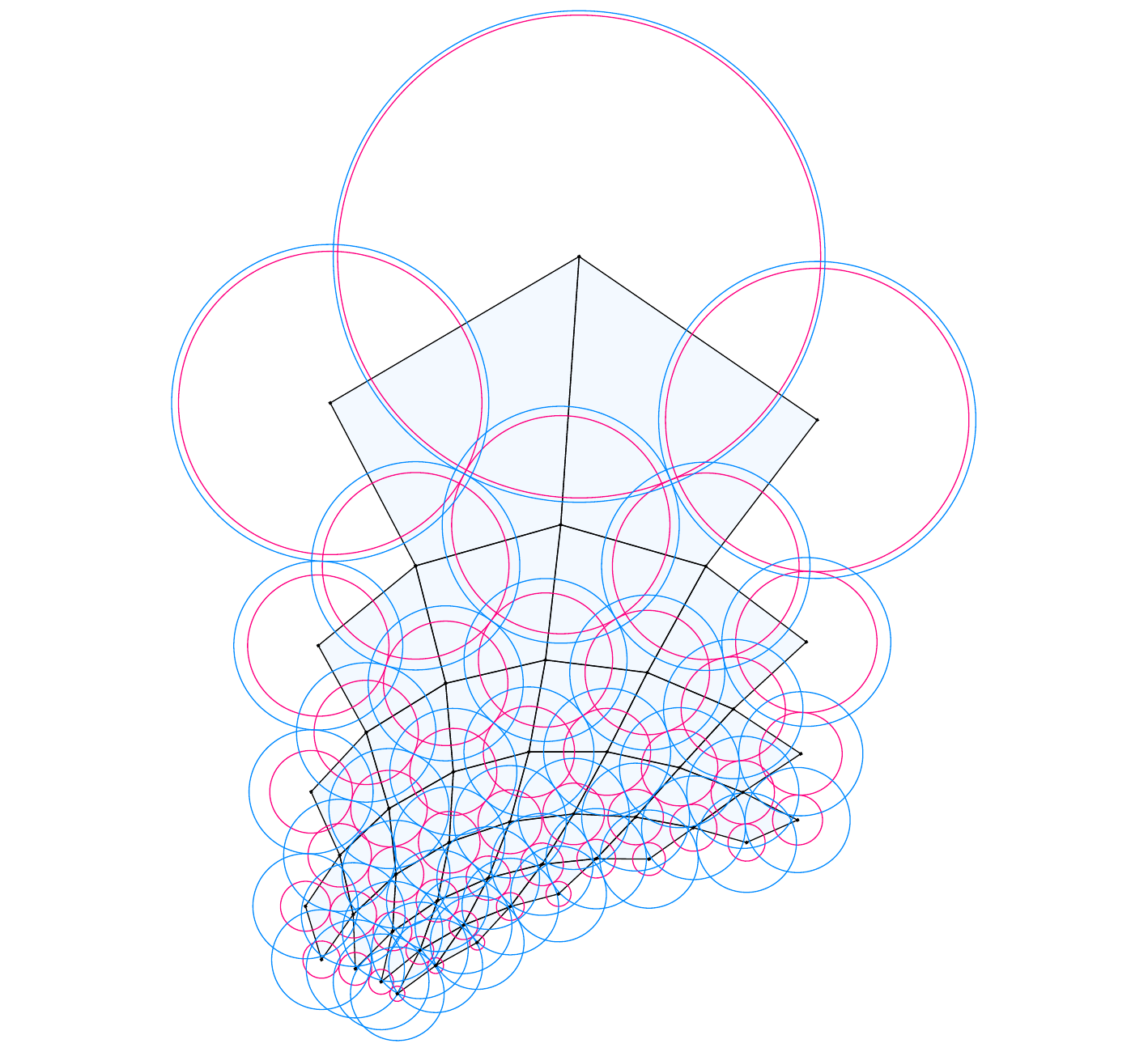}
  \includegraphics[width=.48\linewidth]{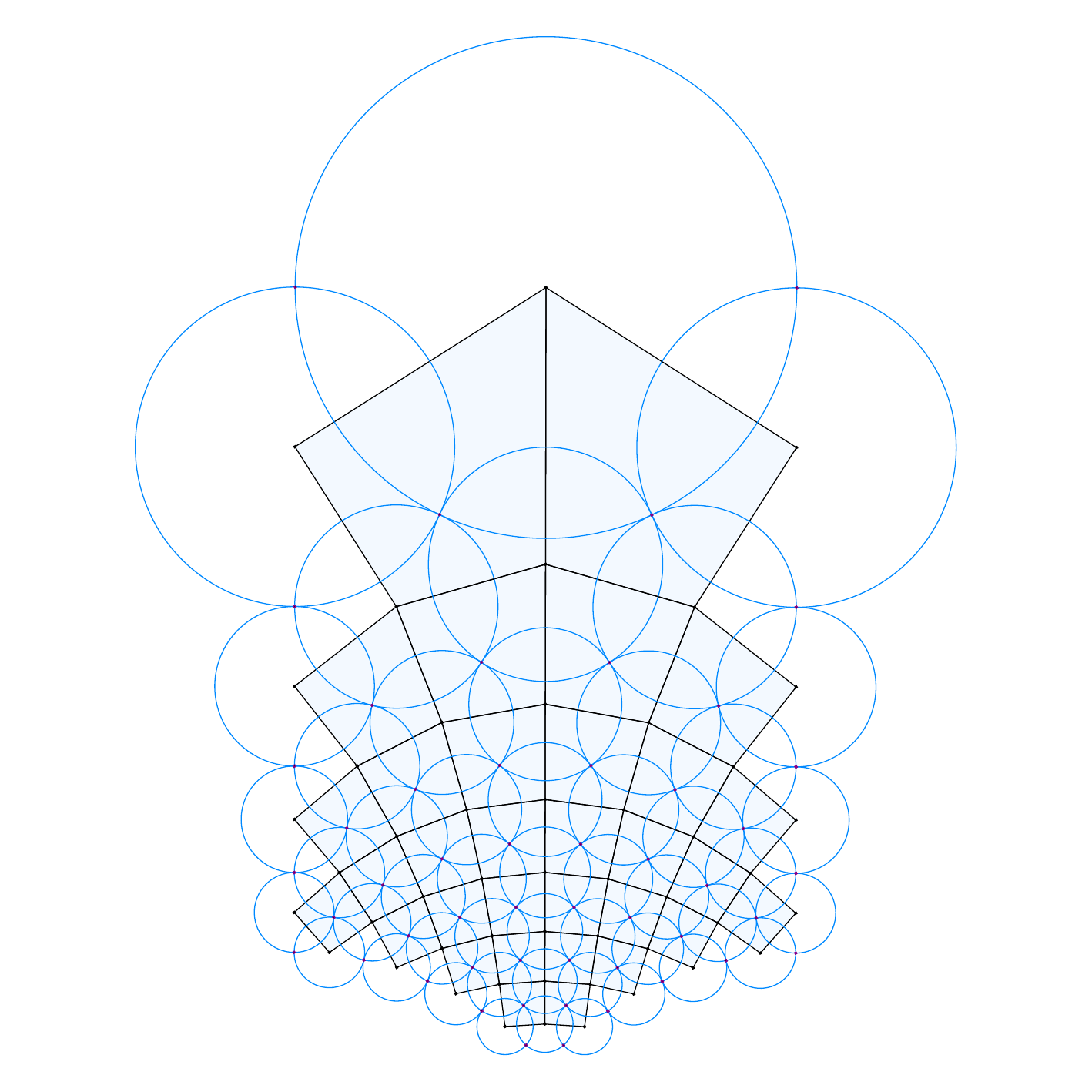}
  \caption{Orthogonal ring patterns interpolating between the circle patterns
  for $z^2$ and its dual pattern for $\log z$.}%
  \label{fig:z2_log}
\end{figure}
\endgroup

\section{Variational description}
\label{sec:variational-description}

The construction of a ring pattern is very similar to the construction of an
orthogonal circle pattern since the equations at the interior vertices are the
same (see Thm.~\ref{thm:ring_pattern}).  For (not necessarily orthogonal)
circle patterns there exists a convex variational
principle~\cite{bobenko_springborn_variational}. 
%In this section we will consider subcomplex~$G$ of~$\Z^2$ whose boundary
%consists of zigzag edges only as shown in Fig.~\ref{fig:variational}.  
For planar orthogonal circle patterns the functional is given in terms of the logarithmic
radii by:
\newcommand{\imli}{\operatorname{Im}\operatorname{Li}_2}
\begin{multline*}
  S(\rho) =  
  \sum_{v_i \edge v_j} 
  \left(
    \imli(i e^{\rho_j - \rho_i}) + 
    \imli(i e^{\rho_i - \rho_j}) -
    \frac{\pi}{2}(\rho_i + \rho_j)
  \right)
  +
  \sum_{v_i}\Phi_i \rho_i,
\end{multline*}
where the first sum is taken over all edges and
the second sum over all vertices of~$G$, 
$\operatorname{Li}_2$ is the dilogarithm
function,  $\imli(i e^x)=\int_{-\infty}^x \arctan e^u du$. 

This functional is invariant with respect to the shift 
\begin{equation}
\label{eq:shift}
\rho_i\to\rho_i +h,\quad \forall i
\end{equation} 
if and only if
\begin{equation}
\label{eq:shift_invariance}
\sum_i \Phi_i=\pi |E(G)|,
\end{equation}
where $|E(G)|$ is the number of edges of $G$. 
The critical points are given by
\begin{equation}
\label{eq:critical}
\frac{\partial S}{\partial \rho_i}=\Phi_i+\sum_{j: v_j{\edge}v_i}\left( 2\arctan (e^{\rho_i-\rho_j})-\pi\right) =0.
\end{equation}
The second derivative 
$$
D^2 S=\sum_{v_i \edge v_j} \frac{1}{\cosh(\rho_i-\rho_j)}(d\rho_i-d\rho_j)^2
$$
is positive for all variations different from (\ref{eq:shift}).

Denote by  $V_B$ the set of boundary vertices of $G$, i.e. the vertices with less then four neighbors. 
For simplicity consider ring patterns with positively oriented rings for all boundary vertices, i.e. on $V_B$ the function $\rho$ 
takes positive values. 
Equations (\ref{eq:critical}) with
\begin{equation}
  \label{eq:neumann}
    \Phi_{i} = 
    \begin{cases}
      2\pi  & \text{for interior vertices}\\
      \Theta_i  & \text{for (positively oriented) boundary vertices}.
    \end{cases}
\end{equation} 
coincide with the orthogonal ring patterns equations (\ref{eq:closure}).
%
%To determine a solution we can either prescribe \emph{Dirichlet boundary
%conditions} by prescribing the $\rho$'s for the boundary rings or
%\emph{Neumann boundary conditions} by specifying the angles $\Theta_i$ for 
%the boundary vertices. 

\begin{proposition} 
Orthogonal ring patterns can be obtained as solutions of the following boundary valued problems:
\begin{itemize}
\item (Dirichlet boundary conditions)
For any choice of prescribed radii $\rho:V_B\to {\mathbb R}_+$ of boundary rings there exists a unique orthogonal 
ring pattern $\mathcal R$.
\item (Neumann boundary conditions) For any choice of  boundary cone angles $\Theta:V_B\to {\mathbb R}_+$ 
satisfying (\ref{eq:shift_invariance}) there exists a one parameter family of orthogonal ring patterns ${\mathcal R}_h$. 
The parameter $h$ is given by the shift (\ref{eq:shift}). 
There exists $h_0$ such that for all $h>h_0$ all boundary rings of the ring pattern ${\mathcal R}_h$ are positively oriented.  
\end{itemize}
\end{proposition}
\begin{proof}
The existence and uniqueness of the boundary valued problems for orthogonal ring patterns can be treated exactly in the same way 
as for circle patterns. The later problems in a more general case were investigated in \cite{bobenko_springborn_variational}.  
The existence and uniqueness for ring patterns follow from the convexity of the functional  $S(\rho)$, for all variations different from (\ref{eq:shift}). This also gives a way to compute the ring patterns by minimizing the functional. For the Neumann boundary valued problem one varies $\rho$'s at all vertices $V(G)$. The condition (\ref{eq:shift_invariance}) implies that the solutions possess the symmetry group (\ref{eq:shift}) described in Section~\ref{sec:circle_patterns}. 
\end{proof}  

An example of solution of a Neumann boundary value problem is presented in Fig.~\ref{fig:variational}). 
Here for all interior vertices $\Phi=2\pi$, for all boundary and not corner vertices $\Phi_i=\Theta_i=\pi$. 
Four angles $\Phi_i=\Theta_i$ at corner vertices of the quadrilateral should sum up to $2\pi$. 
One can easily check that the last condition implies (\ref{eq:shift_invariance}).

\begin{figure}[t]
  \centering
  \includegraphics[width=.9\linewidth]{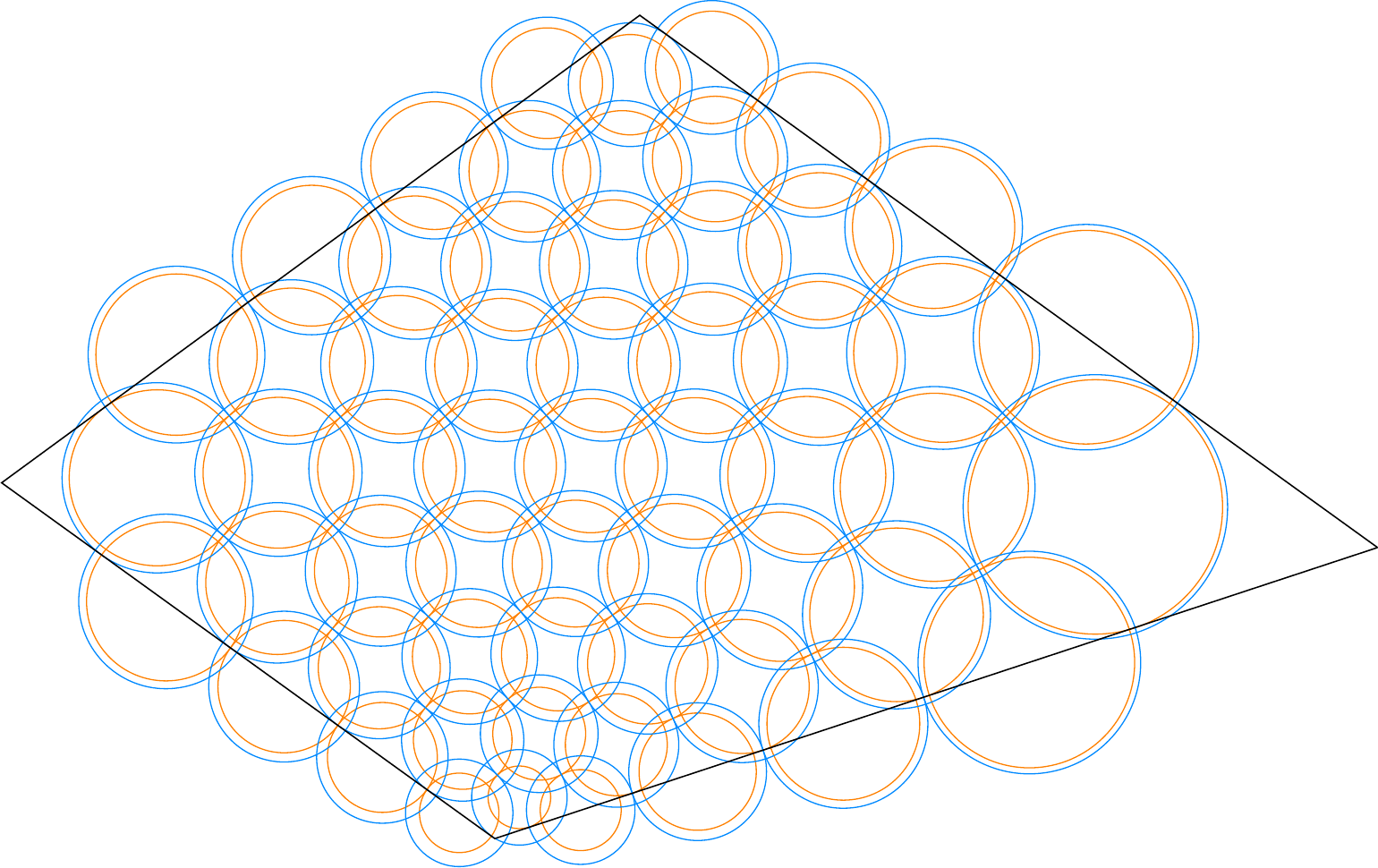}
  \caption{An orthogonal ring pattern computed using the variational
    principle with Neumann boundary conditions. The prescribed angles are
    $\pi$ for the boundary vertices of degree $2$. The shape is governed by the 
    four angles, with the sum $2\pi$, prescribed for the
  four corner boundary vertices of degree~$1$.}%
  \label{fig:variational} 
\end{figure}

\subsection*{Acknowledgements}
We thank Boris Springborn for fruitful discussions on circle patterns and 
variational principles and Nina Smeenk for the support in developing the
software for creating the figures. We also thank the anonymous referee for 
valuable suggestions which improved the presentation.

\bibliographystyle{abbrv}
\bibliography{main}

\end{document}